\documentclass[a4paper]{amsart}
\usepackage{amsfonts}
\usepackage{amssymb}
\usepackage{enumerate}
\usepackage{url,bm}
\usepackage{color}      
\usepackage{graphicx}

\allowdisplaybreaks

\newcommand{\R}{\mathbb{R}}
\newcommand{\N}{\mathbb{N}}
\newcommand{\C}{\mathbb{C}}
\newcommand{\Z}{\mathbb{Z}}
\newcommand{\Tor}{\mathbb{T}}

\newcommand{\T}{\mathcal{T}}
\newcommand{\X}{\mathcal{X}}
\newcommand{\A}{\mathcal{A}}

\newcommand{\bPhi}{\bm{\varPhi}}
\newcommand{\calH}{\mathcal{H}}
\newcommand{\calF}{\mathcal{F}}
\newcommand{\calD}{\mathcal{D}}
\newcommand{\whG}{\widehat{\Gamma}}

\newcommand{\G}{\Gamma}
\newcommand{\W}{\Omega}
\newcommand{\w}{\omega}
\newcommand{\g}{\gamma}
\newcommand{\al}{\alpha}

\newcommand{\supp}{\textrm{supp}}
\newtheorem{theorem}{Theorem}[section]
\newtheorem{lemma}[theorem]{Lemma}
\newtheorem{corollary}[theorem]{Corollary}
\newtheorem{proposition}[theorem]{Proposition}

\def\clspan{{\overline{\mathrm{span}}}}
\newcommand{\wh}[1]{\widehat{#1}}
\newcommand{\esssup}{\mathop{\rm ess\,sup}}

\theoremstyle{remark}
\newtheorem{remark}[theorem]{Remark}
\newtheorem{example}[theorem]{Example}

\theoremstyle{definition}
\newtheorem{definition}[theorem]{Definition}


\title[The Zak transform and the structure of $(\G,\sigma)$-spaces]{The Zak transform and the structure of spaces invariant by the action of an LCA group}
\author{D. Barbieri, E. Hern\'andez, V. Paternostro}

\begin{document}
\address{\textrm{(D. Barbieri)}
Universidad Aut\'onoma
de Madrid, 28049 Madrid, Spain}
\email{davide.barbieri@uam.es}

\address{\textrm{(E. Hern\'andez)}
Universidad Aut\'onoma
de Madrid, 28049 Madrid, Spain}
\email{eugenio.hernandez@uam.es}

\address{\textrm{(V. Paternostro)}
Departamento de Matem\'atica, Facultad de Ciencias Exactas y Naturales, Universidad  de Buenos Aires, Argentina and IMAS-CONICET, Consejo  Nacional de Investigaciones Cient\'ificas y T\'ecnicas, Argentina}
\email{vpater@dm.uba.ar}

\begin{abstract}
We study closed subspaces of $L^2(\X)$, where $(\X, \mu)$ is a $\sigma$-finite measure space, that are invariant under the unitary representation associated to a measurable  action of a discrete countable LCA group $\G$ on $\X$.
We provide a complete description  for these spaces in terms of range functions and a suitable generalized Zak transform. As an application of our main result, we prove a characterization of frames and Riesz sequences in $L^2(\X)$ generated by the action of the unitary representation under consideration  on a countable set of functions in $L^2(\X)$.
Finally, closed subspaces of $L^2(G)$, for $G$ being an LCA group, that are invariant under translations by elements on a closed subgroup $\G$ of $G$ are studied and characterized. The results we obtain for this case are applicable to  cases where  those already proven in \cite{BR14, CP10} are not.
\end{abstract}

\maketitle

\section{Introduction}

Closed spaces that are invariant under a certain kind of unitary
operators are of  particular
importance in several areas of Harmonic Analysis such as  wavelets,
spline systems, Gabor systems and approximation theory
\cite{AG01,Gro01,HW96, Jan98, Mal89}. As a particular case of these
spaces we have the so-called {\it shift-invariant spaces} which are
closed subspaces of $L^2(\R^n)$ invariant under all integer
translations. The structure of shift-invariant spaces has been
intensively studied in \cite{Bow00, dBDVR94a, dBDVR94b, RS95, RS97}.
Recently, shift-invariant spaces have been considered in the context
of locally compact abelian (LCA) groups. That is,  when $G$ is an
LCA group, a shift-invariant space is a closed subspace of  $L^2(G)$
that is invariant under translations by elements of a closed
subgroup $\Gamma$ of $G$. The case when  $\Gamma$ is a discrete
countable subgroup of $G$ with $G/\Gamma$ compact has been treated
in \cite{CP10, KR10}, and extended to the case of $\Gamma$ closed
and $G/\Gamma$ compact in \cite{BR14}.

The aim of this paper is twofold. First of all, we will give a
comprehensive description of the structure of closed spaces that are
invariant by a class of operators that generalize translations, namely unitary representations of
discrete groups arising from  measurable group actions on measure
spaces. Secondly, we will see how a similar  approach allows
us to treat translations by elements of a closed
subgroup without any other assumption on it.  Both cases
are treated with a similar approach by using a generalized Zak
transform, properly adapted to each case.

More precisely, in the first part of the paper we will consider, when $(\X, \mu)$ is a  $\sigma$-finite
measure space, closed subspaces of $L^2(\X)$  invariant under the unitary representation associated to a measurable  action of a discrete countable LCA group
$\Gamma$ on $\X$. A measurable  action of $\Gamma$ on $\X$ is a measurable map $\sigma : \Gamma\times \X \rightarrow \X$ that satisfies
$\sigma_\gamma (\sigma_{\gamma'} (x)) = \sigma_{\gamma + \gamma' }(x)$ for all $\gamma, \gamma' \in \Gamma$, all $x \in \X$ and
$\sigma_e (x) = x$ for all $x\in \X,$ where $\sigma_\gamma : \X \rightarrow \X$ is given by $\sigma_\gamma(x) = \sigma (\gamma, x).$
The action is said to be quasi-$\Gamma$-invariant if there exists a measurable function $J_\sigma : \Gamma \times \X \rightarrow \R^+$
such that $d\mu(\sigma_\gamma(x))= J_\sigma(\gamma, x) d\mu(x).$ To each quasi-$\Gamma$-invariant action $\sigma$ we can
associate a unitary representation $\Pi_{\sigma}$ of $\Gamma$ on $L^2(\X)$ given by
\begin{equation} \label{1.1}
\Pi_{\sigma}(\g)f(x)=J_{\sigma}(-\g,x)^{\frac1{2}}f(\sigma_{-\g}(x)).
\end{equation}
It is with respect to this unitary representation that we consider invariant subspaces of $L^2(\X).$ We say that a closed subspace
$V$ of $L^2(\X)$ is $(\Gamma, \sigma)$-invariant if whenever $f\in V$ we have that $\Pi_\sigma(\gamma) f \in V,\, \forall\,\g\in\G.$

This setting includes the case of shift-invariant subspaces of $L^2(\R^n)$ treated in \cite{Bow00, dBDVR94a, dBDVR94b, RS95, RS97} as well as those considered  in the context of LCA groups \cite{CP10, KR10}.
But moreover, it includes other classical operators
from Analysis such us the dilation operator used in wavelet theory, which  arises  from the dilation action $\sigma : \Z \times \R^n \rightarrow \R^n$
given by $\sigma(j,x) = 2^jx, j\in \Z, x\in \R^n.$ More examples are given in Section 2 (see Example 2.1).

The main issue that arises when trying to address this problem can
be understood by noting that the previous papers where the structure
of shift-invariant spaces is established, such as \cite{Bow00,
CP10}, use heavily the Fourier transform either in $\R^n$ or in the
LCA group $G$, as well as in the subgroup of translations $\Gamma$.
But in our situation, $L^2(\X)$ with $(\X, \mu)$ a measure space,
such a tool is not available. The novelty of our approach is that we
work with a version of the Zak transform adapted to our setting.

 The classical Zak transform is an isometric isomorphism
from $L^2(\R^n)$ onto $L^2(\mathbb T^{2n})$ given, for $f\in
L^2(\R^n)$, by
$$ Zf(x,\xi) = \sum_{k\in \Z^n} f(x+k) e^{-2\pi i k \xi}\,, \qquad
(x,\xi)\in \mathbb T^{2n}\,.$$ The history of the Zak transform is
well described in Chapter 8 of \cite{Gro01}: \textit{The Zak
transform was first introduced and used by Gelfand \cite{Gel50} for
a problem in differential equations. Weil \cite{Wei64} defined this
transform on arbitrary locally compact abelian groups with respect
to arbitrary closed subgroups. \dots Subsequently, the Zak transform
was rediscovered several times, notably by Zak \cite{Zak67, Zak68}
for a problem in solid state physics and Brezin \cite{Bre70} for
differential equations. In representation theory and in abstract
harmonic analysis $Z$ is often called the Weil-Brezin map, but in
applied mathematics and signal analysis it has become customary to
refer to $Z$ as the Zak transform. \dots The popularity of the Zak
transform in engineering seems largely due to Janssen's influential
survey article \cite{Jan88}.} 

 Properties of the Zak transform, as well as
applications in pure and applied mathematics can be found in
\cite{Jan82, Jan88}, \cite[Chapter 8]{Gro01}, \cite{Gla01, Gla04,
HSWW10, HSWW11, Wil11}.

In our setting, the generalized Zak transform is defined  in the following way: for $\psi \in L^2(\X)$,
$\al\in\widehat{\G}$ and $x\in\X$,
$$Z_{\sigma}[\psi](\alpha)(x):=\sum_{\g\in\Gamma}
 \left[(\Pi_{\sigma}(\g)\psi)(x)\right]X_{-\g}(\al),$$
 where  $X_{\g}$ denotes the character on $\widehat \Gamma$ associated to $\gamma \in \Gamma.$

The importance of such generalized Zak transform in the study of invariant subspaces of $L^2(\X)$ has been pointed out in
\cite{HSWW10} where it is used to prove results concerning properties of sequences generated by a single
function $\psi \in L^2(\X)$ under the representation  $\Pi_\sigma$ given by (\ref{1.1}).

Our main result characterizes  $(\Gamma,\sigma)$-invariant subspaces of $L^2(\X)$, for  $\sigma$-finite measure space $\X$, in terms
of range functions and the generalized Zak transform. As a consequence, we are able to characterize frames and Riesz sequences in $L^2(\X)$ generated by the action
of the unitary representation given by (\ref{1.1}) on a countable (meaning countable or finite) family of elements in $L^2(\X)$. As easy corollaries of these characterizations we deduce the results obtained  in
\cite{HSWW10} concerning characterizations of cyclic sequences generated by
a single function in terms of the so-called bracket map. In addition, we also show that every $(\Gamma,\sigma)$-invariant subspace of $L^2(\X)$ can be decomposed into an orthogonal sum of $(\Gamma,\sigma)$-invariant subspaces, each of them  generated by a single function whose $\Pi_\sigma$-orbit is a Parseval frame for the space it spans.

In the second part of the paper, we will then focus on   subspaces
of $L^2(G)$, with  $G$ an LCA group, that are invariant under
translations by elements of a closed
subgroup $\G$ of $G$, and still address this problem in terms of a
suitable Zak transform. The subgroup $\G$ is allowed to be
non-discrete and hence the situation does not fit directly in the
setting previously considered in this paper. For this reason, we need
to work with a non-discrete version of the Zak transform and deal
with some technical issues. As a result, also in this case we can
characterize the structure of these spaces in terms of range
functions but now using the non-discrete version of the Zak
transform. Since we  ask $\G$ to be just closed -- not discrete nor
co-compact -- our results include more cases than those covered by
\cite{BR14, CP10, KR10}.

The paper is organized as follows. Section \ref{sec:preliminaries}
is devoted to give precise definitions of the basic objects involved
in the study, while Section \ref{sec-Zak} contains the main
properties of the Zak transform for discrete LCA groups. The main
result, where we  characterize  $(\Gamma, \sigma)$-invariant
subspaces of $L^2(\X)$ in terms of range functions and the
generalized Zak transform is given in Section \ref{sec:main}.  The
proof of it uses the characterization of multiplicative invariant
subspaces developed in \cite{BR14},  which is summarized, for the
readers convenience, in Subsection \ref{sec-MI}. Applications of
this result are given in Section \ref{sec:frame-riesz} where we
characterize frames and Riesz sequences in $L^2(\X)$ generated by
the action of $\Pi_\sigma$ on a countable set of $L^2(\X)$. Finally,
in Section \ref{sec:TI-spaces} we consider the case of an LCA second
countable group $G$ and provide a characterization of subspaces of
$L^2(G)$ that are invariant under translations by elements
of a closed subgroup $\G$ of $G$.

{\bf Note}: Once this manuscript was submitted for publication we learned of the paper by J. Iverson \cite{Ive14} where similar results are proved independently. In \cite{Ive14}, the results we  prove in Section 6 for closed subspaces of $L^2(G)$ invariant under translations of by a closed subgroup $\G$ of an LCA group $G$, are extended using similar techniques, to the case  of closed abelian  subgroups $\G$ of a locally compact group $G$. On the other hand, in Sections 4 and 5 we obtain results for closed subspaces of a general $L^2(\X)$ space invariant under the action of a discrete LCA group $\G$, which are not included in \cite{Ive14}.  
Also we learned of the paper by S. Saliani \cite{Sal14} on linear independence in spaces invariant under the action of LCA groups, which appeared almost in the same days and contains  ideas on similar topics. 

\

{\bf Acknowledgement}: D. Barbieri was supported by a Marie Curie Intra European Fellowship (Prop. N. 626055) within the 7th European Community Framework Programme.
 D. Barbieri and E. Hern\'andez were supported by Grant MTM2013-40945-P (Ministerio de Econom\'ia y Competitividad, Spain).
 E. Hern\'andez was supported by Grant MTM2010-16518 (Ministerio de Econom\'ia y Competitividad, Spain) 
 and
 V. Paternostro  was supported by a fellowship for postdoctoral researchers
from the Alexander von Humboldt Foundation.\\
The authors want to thank the referee for her/his comments which helped us to improve the manuscript. 

\

\

\section{Preliminaries and notation}\label{sec:preliminaries}

\subsection{Basic on LCA groups}\label{sec:LCA-basics}
Let $\Gamma$   be a locally compact abelian (LCA) group  with operation written additively.
By $\widehat{\G}$ we denote  the dual group of $\G$, that is, the set of continuous characters on $\G$. For $\g\in \G$ and $\al\in\widehat{\G}$ we use the notation $(\g,\al)$ for the complex value that $\al$ takes at  $\gamma$.
For each $\g\in\G$, the corresponding character on $\widehat{\G}$ is denoted by $X_\g$, where  $X_\g:\widehat{\G}\to\C$ and  $X_\g(\al)=(\g,\al)$.

We denote  the Haar measures on $\G$ and $\widehat{\G}$ by $m_\G$ and $m_{\widehat{\G}}$  respectively.
The Fourier transform of a Haar integrable function $f$ on $\G$, is the function
$\widehat{f}$  on $\widehat{\G}$ defined by $\widehat{f}(\al)=\int_\G f(\g)(-\g,\al)\,dm_\G(\g)$.
We fix the Haar measures $m_\G$ and $m_{\widehat{\G}}$  such that the inversion formula holds (for details see \cite[Section 1.5]{Rud62}). Therefore, the Fourier transform on $L^1(\G)\cap L^2(\G)$ extends to a unitary operator from $L^2(\G)$ to $L^2(\widehat{\G})$ that we  denote equally either  by $\wedge$ or by $\mathcal{F}_\G$.

The Fourier transform satisfies the following well known property: for any $\g\in\G$,
\begin{equation}\label{eq:translates-fourier}
\widehat{f(\cdot-\g)}=(-\g, \cdot)\widehat{f(\cdot)}.
\end{equation}
As usual, when $\G$ is discrete, we choose $m_\G$ to be counting measure. Then, the inversion formula holds with $m_{\widehat{\G}}$ normalized such that $m_{\widehat{\G}}(\widehat{\G} )=1$. With this normalization of the Haar measure and supposing that $\G$ is countable, $\{X_\g\}_{\g\in\G}$ turns out to be an orthonormal basis for $L^2(\whG)$.

\subsection{LCA groups acting on $\sigma$-finite  measure spaces}\label{sec-action}
We now state the precise definition of the actions we shall work with. This type of actions, called {\it quasi-$\G$-invariant} actions,  were previously considered in \cite{HSWW10} within the abelian setting and
in \cite{BHP14} for countable discrete groups that can be non-abelian.
Fix $\G$ a countable discrete LCA group.
Let $(\X,\mu)$ be  a $\sigma$-finite measure space and $\sigma:\Gamma\times\X\to\X$ an  action satisfying the following conditions:
\begin{enumerate}
 \item [(i)] for each $\g\in\Gamma$ the map $\sigma_{\g}:\X\to\X$ given by $\sigma_{\g}(x):=\sigma(\g,x)$ is measurable;
\item [(ii)] $\sigma_{\g}(\sigma_{\g'}(x))=\sigma_{\g+\g'}(x)$, for all $\g,\g'\in\Gamma$ and for all $x\in\X$;
\item [(iii)] $\sigma_{e}(x)=x$ for all $x\in\X$, where $e$ is the identity of $\Gamma$.
\end{enumerate}
As $\G$ is discrete and countable,  $\sigma:\Gamma\times\X\to\X$ turns out to be a measurable action,
since if $E\subseteq \X$ is measurable set, then $\sigma^{-1}(E)=\bigcup_{\g\in\G}[\{\g\}\times\sigma_\g^{-1}(E)]$
is measurable in $\Gamma \times \X$.

The action $\sigma$ is said to be {\it quasi-$\G$-invariant} if there exists  a measurable positive function $J_{\sigma}:\Gamma\times\X\to\R^{+}$, called {\it Jacobian of $\sigma$}, such that
$d\mu(\sigma_{\g}(x))=J_{\sigma}(\g,x)d\mu(x).$ As a consequence of  the properties of $\sigma$  and the definition of the Jacobian  we have:
\begin{equation}\label{jacobian}
 J_{\sigma}(\g_1+\g_2,x)=J_{\sigma}(\g_1,\sigma_{\g_2}(x))J_{\sigma}(\g_2,x), \quad\forall \,\g_1,\g_2\in\G, \,\, x\in\X.
\end{equation}

To each quasi-$\G$-invariant action $\sigma$ we can associate a unitary representation $\Pi_{\sigma}$ of $\Gamma$ on $L^2(\X)$ given by
$$\Pi_{\sigma}(\g)\psi(x)=J_{\sigma}(-\g,x)^{\frac1{2}}\psi(\sigma_{-\g}(x)).$$
From \eqref{jacobian} it follows that $\Pi_\sigma$ is a unitary operator. To prove that $\Pi_\sigma$ is a representation see \cite[Section 4.2]{BHP14}.

The action $\sigma$ has the {\it tiling property } if there exists a measurable set $C\subseteq \X$ such that
$\mu(\sigma_\g(C)\cap\sigma_{\g'}(C))=0$ if $\g\neq\g'$ and $\mu(\X\setminus\bigcup_{\g\in\G}\sigma_\g(C))=0$.

We provide here relevant examples of concrete  quasi-$\G$-invariant actions satisfying the tiling property. In particular, they include the usual operators used in Analysis such us translations, dilations and those arising from the action of the shear matrix. Note that, the tiling set $C$ necessarily has positive measure but,  as the examples show, it  can  have either finite or infinite measure.
\begin{example}\noindent
 \begin{itemize}
  \item Let us consider the group $\Z$ acting on $\R^2$ by
  $$\sigma_1(k, (x,y))=\begin{pmatrix}
                        1&0\\
                        k&1
                    \end{pmatrix}\begin{pmatrix}
                      x\\
                      y
                    \end{pmatrix}=(x, kx+y).$$
The action $\sigma_1$ satisfies conditions i-iii and $J_1(k, (x,y))=1$. Moreover, $\sigma_1$ has the tiling property with  tiling set $C=\{(x,y)\in\R^2: 0\leq y\leq x\}\cup\{(x,y)\in\R^2: x\leq y\leq 0\}$.
The associated unitary representation on $L^2(\R^2)$ is $\Pi_{\sigma_1}(k)\psi(x,y)=\psi(x, -kx+y)$.

\item A different action of $\Z$ on $\R^2$ is given by
$$\sigma_2(j, (x,y))=\begin{pmatrix}
                        2^j&0\\
                        0&2^{j/2}
                    \end{pmatrix}\begin{pmatrix}
                      x\\
                      y
                    \end{pmatrix}=(2^jx, 2^{j/2}y).$$
Also in this case $\sigma_2$ satisfies conditions i-iii but we have that $J_2(j, (x,y))=2^{3j/2}$. A tiling set for $\sigma_2$ is $C=C_1\setminus C_2$ where $C_1=[-1, 1]^2$ and $C_2=[-1/2, 1/2]\times [-1/\sqrt{2}, 1/\sqrt{2}]$. The associated unitary representation on $L^2(\R^2)$ is $\Pi_{\sigma_2}(j)\psi(x,y)=2^{-3j/4}\psi(2^{-j}x, 2^{-j/2}y)$.

\item The group $\Z^n$ acts on $\R^n$ by translations $\sigma_3(m, x)=m+x$. For this action $J_3(m,x)=1$, a tiling set is $C=[-1/2,1/2]^n$ and the unitary representation associated on $L^2(\R^n)$ to $\sigma_3$ is $\Pi_{\sigma_3}(m)\psi(x)=\psi(x-m)$.

\end{itemize}
\end{example}

\subsection{$(\G, \sigma)$-invariant spaces on $L^2(\X)$}
Given  $\sigma$ a quasi-$\G$-invariant action of $\G$ on $\X$ we can define $(\G, \sigma)$-invariant spaces in $L^2(\X)$ as spaces that are invariant by the operators $\Pi_\sigma(\g)$, $\g\in\G$. More precisely, we have:
\begin{definition}
A closed subspace $V$ of $L^2(\X)$ is said to be {\it $(\G, \sigma)$-invariant} if
$$f\in V\Longrightarrow \Pi_{\sigma}(\g)f\in V, \,\, \textrm{ for any }\,\,\g\in\Gamma.$$

For any subset $\A\subseteq  L^2(\X)$ we define
\[
S_{\sigma}^{\G}(\A)= \overline{\mbox{span}}\{\Pi_{\sigma}(\g)\phi\colon \phi\in\A, \g\in\Gamma\}\,\,\textrm{ and }\,\,
E_{\sigma}^{\G}(\A)= \{\Pi_{\sigma}(\g)\phi\colon \phi\in\A, \g\in\Gamma\}.
\]

We call $S_{\sigma}^{\G}(\A)$ the $(\G, \sigma)$-invariant space  generated by $\A$.
If $V=S_{\sigma}^{\G}(\A)$ for some finite set $\A$  we say that $V$ is a {\it finitely generated} $(\G, \sigma)$-invariant space, and a
{\it principal or cyclic } $(\G, \sigma)$-invariant space if $\A$ has only a single function.
\end{definition}

When $L^2(\X)$ is separable
each $(\G, \sigma)$-invariant space $V$ is of the form $V=S_{\sigma}^{\G}(\A)$ for some countable (meaning finite or countable) set $\A\subseteq  L^2(\X)$.

\begin{remark}
 By \cite[Proposition 3.4.5]{Coh93}, the space $L^2(\X)$ is separable if and only if $\mu$ is a $\sigma$-finite measure and the $\sigma$-algebra on $\X$ is countably generated, meaning that it is generated by countable family of subsets of $\X$.
 From now on, we will always assume that $L^2(\X)$ is separable, even  without mentioning it.
\end{remark}

\subsection{ Vector valued spaces}
Let $(\Omega, \nu)$ be a $\sigma$-finite measure space and let $\mathcal{H}$ be a separable Hilbert space.
The vector valued space $L^2(\Omega, \mathcal{H})$ is the space of measurable functions $\Phi:\Omega\to\mathcal{H}$ such that $\|\Phi\|^2=\int_{\Omega}\|\Phi(\omega)\|^2_{\mathcal{H}}\,d\nu(\omega)<+\infty$. The inner product in $L^2(\Omega, \mathcal{H})$ is given by $\langle \Phi, \Psi\rangle=\int_{\Omega}\langle \Phi(\omega), \Psi(\omega)\rangle_{\mathcal{H}}\,d\nu(\omega)$.

A function $\Phi:\Omega\to\mathcal{H}$ is called {\it simple} if there exist $x_1, \ldots, x_n\in\calH$ and measurable subsets $E_1, \ldots, E_n$ of $\Omega$ such that $\Phi=\sum_{j=1}^nx_j\chi_{E_j}$, where $\chi_{E_j}$ denotes the characteristic function of $E_j$.
When $\Omega$ is of finite measure, the set of simple functions is dense in $L^2(\Omega, \mathcal{H})$. Since we could not find a reference for this fact, we provide a proof  in the Appendix. More details about the subject can be found in
\cite{DU77}.

\subsection{Frames and Riesz bases}
We briefly  recall the definitions of frame and Riesz basis. For a detailed exposition on this subject we refer to \cite{Chr03}.

Let $\mathcal{H}$ be a separable Hilbert space, $I$ be  a finite or countable index set and   $\{f_i\}_{i\in I}$ be a sequence in $\mathcal{H}.$
The sequence $\{f_i\}_{i\in I}$ is said to be  a {\it frame} for $\mathcal{H}$ if there exist $0<A\leq B<+\infty$ such that
\begin{equation*}
A\|f\|^2\leq \sum_{i\in I} |\langle f,f_i\rangle|^2\leq B\|f\|^2
\end{equation*}
for all $f\in\mathcal{H}$. The constants $A$ and $B$ are called {\it frame bounds}. When $A=B=1$,  $\{f_i\}_{i\in I}$ is called {\it  Parseval frame}.

The sequence $\{f_i\}_{i\in I}$ is said to be  a {\it Riesz basis} for $\mathcal{H}$ if it is a complete system in $\calH$ and if there exist $0<A\leq B<+\infty$ such that
\begin{equation*}
A\sum_{i\in I}|a_i|^2\leq \|\sum_{i\in I} a_if_i\|^2\leq B\sum_{i\in I}|a_i|^2
\end{equation*}
for all sequences $\{a_i\}_{i\in I}$ of finite support.

The sequence $\{f_i\}_{i\in I}$ is a {\it frame (or Riesz)
sequence}, if it is a frame (or Riesz basis) for the Hilbert space
it spans, namely $\clspan\{f_i:\, i\in I\}$.

\section{A Zak transform associated to quasi-$\G$-invariant actions}\label{sec-Zak}

In \cite{HSWW10} the authors introduced a definition of the Zak transform associated to quasi-$\G$-invariant action $\sigma$  of a discrete and countable  LCA group $\G$. Recently, in \cite{BHP14}, an analogous  noncommutative Zak transform was  introduced for non-abelian groups. In this section we will make use of the definition given in \cite{HSWW10} and see how the Zak transform allows us to think of $L^2(\X)$ as a vector valued space. The Zak transform  combined with range functions  will allow us to prove a characterization of  $(\G, \sigma)$-invariant spaces similar to those previously given in \cite{Bow00,  BR14, CP10, CMO14}.

Assume $\sigma $ is a quasi-$\G$-invariant action of $\G$ on $\X$ satisfying the tiling property with tiling set $C$.
Let $\Pi_\sigma$ be its associated unitary representation of $\G$ on $L^2(\X)$.
Given $\psi\in L^2(\X)$ we can define the function $\psi_\sigma:\G\times\X\to\C$ as
\begin{equation}\label{def-T-sigma}
 \psi_\sigma(\g,x):=\Pi_\sigma(\g)\psi(x).
\end{equation}
For a fixed $x\in\X$ the function $\psi_\sigma(\cdot,x)$ can be understood as the evolution of $\psi$ along the orbit of $x$, $\{\sigma_\g(x):\g\in\G\}$.

\begin{lemma}\label{lemma-L-2-function}
 For every $\psi\in L^2(\X)$, the sequence
 $\psi_\sigma(\cdot,x)=\{\psi_\sigma(\g, x)\}_{\g\in\G}$ belongs to $\ell^2(\Gamma)$ for $\mu$-a.e. $x\in \X$.
\end{lemma}

\begin{proof}
Using the tiling property of the action $\sigma$, we obtain
\begin{align*}
 \|\psi\|^2_{L^2(\X)}&=\int_{\X}|\psi(x)|^2\,d\mu(x) = \sum_{\g\in\Gamma}\int_{\sigma_{\g}(C)}|\psi(x)|^2\,d\mu(x)\\
&=\sum_{\g\in\Gamma}\int_{C}|\psi(\sigma_{\g}(y))|^2J_{\sigma}(\g,y)\,d\mu(y)\\
&=\int_{C}\sum_{\g\in\Gamma}|\psi_\sigma(-\g, y)|^2\,d\mu(y)
=\int_{C}\|\psi_\sigma(\cdot, y)\|^2_{\ell^2(\G)}\,d\mu(y).
\end{align*}
Since $\psi\in L^2(\X)$, it follows that
$\|\psi_\sigma(\cdot, y)\|^2_{\ell^2(\G)}<+\infty$ for $\mu$-a.e. $y\in C$.

Finally, by \eqref{jacobian} it follows that for  $x\in\X$ written as $x=\sigma_{\g_0}(y)$ with $\g_0\in\G$ and $y\in C$
$$\psi_\sigma(\g,x)=\psi_\sigma(\g,\sigma_{\g_0}(y))=J_\sigma(\g_0,y)^{-\frac{1}{2}}\psi_\sigma(\g-\g_0, y).$$
Thus,
$\|\psi_\sigma(\cdot, x)\|^2_{\ell^2(\G)}=J_\sigma(\g_0,y)^{-1}\|\psi_\sigma(\cdot, y)\|^2_{\ell^2(\G)}$ and the lemma follows.
\end{proof}

\begin{definition}[\cite{HSWW10}]\label{def-tau-sigma}
 Given any  $\psi\in L^2(\X)$ define for $\al\in\widehat{\G}$ and  $x\in\X$ such that $\psi_\sigma(\cdot,x)\in\ell^2(\Gamma)$
$$Z_{\sigma}[\psi](\alpha)(x):=\sum_{\g\in\Gamma}
 \left[(\Pi_{\sigma}(\g)\psi)(x)\right](-\g,\al).$$
 We call $Z_\sigma[\psi]$ the {\it (generalized) Zak  transform of $\psi$}.
\end{definition}
Note that $Z_{\sigma}[\psi](\alpha)(x)$ is the Fourier transform of the sequence $\psi_\sigma(\cdot, x)$ evaluated in $\al$; i.e. $Z_{\sigma}[\psi](\alpha)(x):=\widehat{\psi_\sigma(\cdot, x)}(\al).$ By Lemma \ref{lemma-L-2-function}, $Z_{\sigma}[\psi](\alpha)(x)$ is defined for $\mu$-a.e. $x\in\X$ and every $\al\in\whG$.

In the next proposition we show that $L^2(\X)$ is isometrically isomorphic to the vector valued Hilbert space
$L^2(\widehat{\Gamma}, L^2(C))$.
For this vector valued space, the norm and inner product are
$$\|\Phi\|^2=\int_{\widehat{\G}}\|\Phi(\al)\|^2_{L^2(C)}\,dm_{\widehat{\G}}(\al)\quad\textrm{ and} \quad \langle \Phi, \Psi\rangle=\int_{\widehat{\G}}\langle \Phi(\al), \Psi(\al)\rangle_{L^2(C)}\,dm_{\widehat{\G}}(\al),$$
respectively.

\begin{proposition}\label{abelian-tau-isometry}
 The mapping $Z_{\sigma}: L^2(\X) \longrightarrow  L^2(\widehat{\Gamma}, L^2(C))$ defined by
 $$Z_{\sigma}[\psi](\alpha)(x)=\sum_{\g\in\Gamma}
 \left[(\Pi_{\sigma}(\g)\psi)(x)\right](-\g,\al)$$
 is an isometric isomorphism and satisfies
\begin{equation}\label{quasi-periodic}
Z_{\sigma}[\Pi_{\sigma}(\g)\psi]=X_\g Z_{\sigma}[\psi], \quad \forall \,\,\g\in\G,\,\,\psi\in L^2(\X).
\end{equation}
\end{proposition}

\begin{proof}
For each $\psi\in L^2(\X)$, $Z_\sigma[\psi]$ is measurable as a
function defined on $\whG\times C$. Now, proceeding as in    Lemma
\ref{lemma-L-2-function} and by  Plancherel and
Fubini's Theorems we have
\begin{align}\label{eq:zak-well-defined}
\|\psi\|^2_{L^2(\X)}&=\int_{C}\sum_{\g\in\Gamma}|(\Pi_{\sigma}(\g)\psi)(x)|^2\,d\mu(x)\nonumber\\
&=\int_{C}\|\psi_\sigma(\cdot, x)\|_{\ell^2(\G)}^2\,d\mu(x)=\int_{C}\|\widehat{\psi_\sigma(\cdot, x)}\|_{L^2(\widehat{\G})}^2\,d\mu(x)\nonumber\\
&=\int_{C}\int_{\widehat{\G}}|\sum_{\g\in\Gamma}
 \left[(\Pi_{\sigma}(\g)\psi)(x)\right](-\g,\al)|^2\,dm_{\widehat{\G}}(\al)\,d\mu(x)\nonumber\\
 &=\int_{\widehat{\Gamma}}\int_{C}|Z_{\sigma}[\psi](\alpha)(x)|^2\,d\mu(x)\,dm_{\widehat{\G}}(\al).
\end{align}
Therefore, $Z_{\sigma}[\psi](\alpha)\in L^2(C)$ for $m_{\widehat{\G}}$-a.e $\al\in\whG$. Moreover, using again  Fubini's
Theorem, one can see that $Z_\sigma[\psi]$ is measurable as a vector valued function from $\whG$ to $L^2(C)$. Continuing with \eqref{eq:zak-well-defined}, we obtain
$$\|\psi\|^2_{L^2(\X)}=\int_{\widehat{\Gamma}}\|Z_{\sigma}[\psi](\alpha)\|^2_{L^2(C)}\,dm_{\widehat{\G}}(\al)=\|Z_{\sigma}[\psi]\|^2.$$
Thus, $Z_{\sigma}[\psi]\in L^2(\widehat{\Gamma}, L^2(C))$ and $Z_\sigma$ is an isometry.

Let us see now   that $Z_{\sigma}$ is onto. Let $\Phi\in  L^2(\widehat{\Gamma}, L^2(C))$ be a simple function. Then, $\Phi=\sum_{j=1}^nf_j\chi_{E_j}$, where $f_j\in L^2(C)$ and $E_j\subseteq \whG$ is $m_{\whG}$-measurable, for $j=1,\ldots,n$. Since $\{X_\g\}_{\g\in\G}$ is an orthonormal basis for $L^2(\whG)$, we can write
$\chi_{E_j}=\sum_{\g\in\G}a_\g^jX_\g$ where $a^j:=\{a_\g^j\}_{\g\in\G}\in \ell^2(\G)$ for every $j=1,\ldots,n$, and then
$\Phi=\sum_{\g\in\G}\sum_{j=1}^nf_ja_\g^jX_\g$.

For $\g\in\G$ we define  $\psi$ on $\sigma_\g(C)$ as $\psi(\sigma_\g(x)):=J(\g,x)^{-1/2}\sum_{j=1}^nf_j(x)a_\g^j$ for $\mu$-a.e. $x\in C$. Since $C$ is a tiling set, we then have $\psi$ defined on $\X$. Furthermore,
\begin{align*}
 \|\psi\|^2_{L^2(\X)}&=\sum_{\g\in\Gamma}\int_{C}|\psi(\sigma_{\g}(x))|^2J_{\sigma}(\g,x)\,d\mu(x)\\
&=\sum_{\g\in\Gamma}\int_{C}|\sum_{j=1}^nf_j(x)a_\g^j|^2\,d\mu(x)\leq \sum_{\g\in\Gamma}\int_{C}\sum_{j=1}^n|f_j(x)|^2\sum_{l=1}^n|a_\g^l|^2\,d\mu(x)\\
&=\sum_{l=1}^n\|a^l\|_{\ell^2(\G)}^2\sum_{j=1}^n\|f_j\|_{L^2(C)}^2
<+\infty,
\end{align*}
and thus $\psi\in L^2(\X)$. Now for $\al\in\whG$ and $x\in C$ we have
\begin{align*}
 Z_\sigma[\psi](\al)(x)&=\sum_{\g\in\Gamma}\psi(\sigma_{-\g}(x))J(\g, x)^{-1/2}(-\g, \al)\\
 &=\sum_{\g\in\Gamma}\sum_{j=1}^nf_j(x)a_{-\g}^j(-\g, \al)
 =\Phi(\al)(x).
\end{align*}
Therefore, the set of simple functions is contained in the range of
$Z_\sigma$. Since $\G$ is discrete and then $\whG$ compact,
$m_{\whG}(\whG)$ is finite and the set of simple functions is dense
in $L^2(\widehat{\Gamma}, L^2(C))$ (See Proposition
\ref{prop:simples-are-dense} in the Appendix). Hence, due to the
fact that the range of $Z_\sigma$ is closed, we conclude that
$Z_\sigma$ is onto.

To verify that $Z_\sigma$ satisfies \eqref{quasi-periodic} observe that, since $\Pi_\sigma$ is a unitary representation,
and each $\al \in \widehat \Gamma$ is a homomorphism, we have
\begin{align*}
 Z_\sigma[\Pi(\g_0)\psi](\al)(x)&=\sum_{\g\in \Gamma}[\Pi_\sigma(\g)\Pi_\sigma(\g_0)\psi(x) ](-\g, \al)
 =\sum_{\g\in \Gamma}[\Pi_\sigma(\g+\g_0)\psi(x)] (-\g, \al)\\
 &=\sum_{\g\in \Gamma}[\Pi_\sigma(\g)\psi(x) ](-\g+\g_0, \al)=(\g_0,\al)Z_\sigma[\psi](\al)(x). \qedhere
\end{align*}
\end{proof}

\section{Characterization of $(\G, \sigma)$-invariant spaces of $L^2(\X)$}\label{sec:main}

In this section we want to characterize $(\G, \sigma)$-invariant spaces in terms of the generalized Zak transform and range functions.
 Towards this end, we will evoke the general machinery developed by Bownik and Ross in \cite{BR14} where they defined and characterized multiplicatively invariant spaces on $L^2(\Omega, \mathcal{H})$.

\subsection{Multiplicatively invariant spaces in $L^2(\Omega, \mathcal{H})$}\label{sec-MI}
In order to make our exposition self contained, we hereby summarize the material  we need concerning multiplicatively  invariant spaces. See \cite[Section 2]{BR14} for details and proofs.

Fix $(\Omega, \nu)$ a $\sigma$-finite measure space and $\mathcal{H}$ a separable Hilbert space.

A set $D\subseteq L^{\infty}(\Omega)$ is said to be a {\it determining set} for $L^1(\Omega)$ if
for every $f\in L^1(\Omega)$ such that $\int_{\Omega}f(\omega)g(\omega)\,d\nu(\omega)=0\,\,\forall g\in D$, one has $f=0$.
In the setting of Helson \cite{Hel64, Hel86}, a determining set is the set of exponentials with integer parameter, $D=\{e^{2\pi i k\cdot}\}_{k\in\Z}\subseteq L^{\infty}(\mathbb{T})$.

\begin{definition}\label{def-MI}
A closed subspace $M\subseteq  L^2(\Omega, \mathcal{H})$ is {\it multiplicatively  invariant} with respect to the determining set $D$ for $L^1(\Omega)$ (MI space for short) if
$$\Phi\in M\Longrightarrow g\Phi\in M, \,\, \textrm{ for any }\,\,g\in D.$$
For an at most countable   subset $\bm{\varPhi}\subseteq L^2(\Omega, \mathcal{H})$  define
$
M_D(\bm{\varPhi})= \overline{\mbox{span}}\{g\Phi\colon \Phi\in\bm{\varPhi}, g\in D\}.
$
The subspace $M_D(\bm{\varPhi})$ is called the multiplicatively  invariant space  generated by $\bm{\varPhi}$, and we say that $\bPhi$ is a set of generators for $M_D(\bPhi)$.
\end{definition}

MI spaces have been recently characterized in terms of measurable range functions by Bownik and Ross in \cite{BR14}.
We recall that a  {\it range function} is a mapping $J:\Omega\to \{\textrm{closed subspaces of }\mathcal{H}\}$ equipped
with the orthogonal projections $P_J(\omega)$ of $\mathcal{H}$ onto $J(\omega)$. A range function is said to be {\it measurable} if for every $a\in\mathcal{H}$, $\omega\mapsto P_J(\omega)a$ is measurable as a vector valued  function.

\begin{theorem}\cite[Theorem 2.4]{BR14}\label{thm-MI-rango}
Suppose that $L^2(\Omega)$ is separable, so that $L^2(\Omega, \mathcal{H})$ is also separable.
Let $M$ be a closed subspace of $L^2(\Omega, \mathcal{H})$ and $D$ a determining set for $L^1(\Omega)$.
Then, $M$ is an MI space with respect to $D$ if and only if there exists a measurable range function $J$ such that
$$M=\{\Phi\in  L^2(\Omega, \mathcal{H})\,:\, \Phi(\omega)\in J(\omega) \textrm{ a.e. } \omega\in\Omega\}.$$
Identifying range functions that are equal almost everywhere, the correspondence between MI spaces and measurable range functions is one-to-one and onto.

Moreover, when $M=M_D(\bm{\varPhi})$ for some at most countable set $\bm{\varPhi}\subseteq  L^2(\Omega, \mathcal{H})$ the range function associated to $M$ is
$$J(\w)=\clspan\{\Phi(\w):\Phi\in\bm{\varPhi}\}, \quad \textrm{a.e. }\w\in\W.$$
\end{theorem}

\subsection{Characterization of $(\G, \sigma)$-invariant spaces}
To our purpose of characterizing $(\G, \sigma)$-invariant spaces we will connect them with MI spaces. We start by choosing $\Omega=\widehat{\G}$ and $\mathcal{H}=L^2(C)$. We point out that since we assume that $L^2(\X)$ is  separable, so is $L^2(C)$. Furthermore, the space $L^2(\whG)$ is also separable and then $L^2(\whG, L^2(C))$ is separable as well.
In this setting a {\it range function} is  a mapping
\begin{equation}\label{range-function}
 J:\widehat{\Gamma}\longrightarrow \{\textrm{closed subspaces of } L^2(C)\}
\end{equation}
equipped with the orthogonal projections $P_{J}(\alpha):L^2(C)\to J(\alpha)$.

Our main result is:
\begin{theorem}\label{thm-main}
Let $V\subseteq L^2(\X)$ be a closed subspace and let $Z_{\sigma}$ be the Zak transform of Proposition \ref{abelian-tau-isometry}.
Then, the following conditions are equivalent:
\begin{enumerate}
 \item $V$ is a $(\G, \sigma)$-invariant  space;
 \item there exists a measurable  range function $J$ such that
$$V=\big\{f\in L^2(\X):\,\, Z_{\sigma}[f](\alpha)\in J(\alpha)\,\,\textrm{a.e.}\,\,\alpha\in \widehat{\Gamma}\big\}.$$
\end{enumerate}

Identifying  range functions which are equal almost everywhere, the correspondence between
$(\G, \sigma)$-invariant spaces and  measurable range functions is one to one and onto.

Moreover, if $V=S_{\sigma}^{\G}(\A)$ for some countable subset $\A$ of $L^2(\X)$, the measurable  range function $J$
associated to $V$ is given by
$$J(\alpha)=\clspan\{Z_{\sigma} [\phi](\alpha)\,:\,\phi\in\A\}, \qquad \textrm {a.e. }\alpha\in \widehat{\Gamma}.$$
\end{theorem}

The proof of Theorem \ref{thm-main} is based in the next lemma.

\begin{lemma}\label{lemma-connection-MI-spaces}
 Let $V\subseteq L^2(\X)$ be a closed subspace. Then,
 $V$ is a $(\G, \sigma)$-invariant space if and only if $Z_\sigma[V]\subseteq L^2(\widehat{\Gamma}, L^2(C))$  is a MI space with respect to the determining set $D=\{X_\g\}_{\g\in\Gamma}$.
 In particular, $\A\subseteq L^2(\X)$ is a set of generators for the $(\G, \sigma)$-invariant space $V$ if and only if the set $Z_\sigma[\A]$ generates $Z_\sigma[V]$ as a $D$-MI space.
\end{lemma}

\begin{proof}
 First note that,  as a consequence of the uniqueness of the Fourier transform,  the set $D=\{X_\g\}_{\g\in\Gamma}$ is a determining set for $L^1(\widehat{\G})$. Then, the result follows from the fact that the Zak transform is an isomorphism and  by property \eqref{quasi-periodic}.
\end{proof}

\begin{proof}[Proof of Theorem \ref{thm-main}]
 Let us first show ${\it (1)\Rightarrow(2)}$. By Lemma \ref{lemma-connection-MI-spaces}, $Z_\sigma[V]$ is a $D$-MI space in $L^2(\widehat{\Gamma}, L^2(C))$. Then by Theorem \ref{thm-MI-rango} there exists a measurable
 range function  $J$  such that
 $Z_\sigma[V]=\big\{\Phi\in L^2(\widehat{\Gamma}, L^2(C)):\,\, \Phi(\alpha)\in J(\alpha)\,\,\textrm{a.e.}\,\,\alpha\in \widehat{\Gamma}\big\}.$
 This, together with Proposition \ref{abelian-tau-isometry} implies that $V=\big\{f\in L^2(\X):\,\, Z_{\sigma}[f](\alpha)\in J(\alpha)\,\,\textrm{a.e.}\,\,\alpha\in \widehat{\Gamma}\big\},$ proving ${\it (2)}$.

 To prove ${\it (2)\Rightarrow(1)}$ observe that if $V=\big\{f\in L^2(\X):\,\, Z_{\sigma}[f](\alpha)\in J(\alpha)\,\,\textrm{a.e.}\,\,\alpha\in \widehat{\Gamma}\big\},$ for some measurable range function $J$, then $Z_\sigma[V]=\big\{\Phi\in L^2(\widehat{\Gamma}, L^2(C)):\,\, \Phi(\alpha)\in J(\alpha)\,\,\textrm{a.e.}\,\,\alpha\in \widehat{\Gamma}\big\}$.  This  implies that $Z_\sigma[V]$ is a $D$-MI space. Indeed, if $\Phi\in Z_\sigma[V]$, $\Phi(\alpha)\in J(\alpha)$ a.e. $\al\in\widehat{\G}$ and thus, since $J(\al)$ is a subspace, $X_\g(\al) \Phi(\alpha)\in J(\alpha)$ a.e. $\al\in\widehat{\G}$. Therefore, $X_\g\Phi\in Z_\sigma[V]$. Then, Lemma \ref{lemma-connection-MI-spaces} guarantees that $V$ is a $(\G, \sigma)$-invariant space.

 The bijective correspondence between measurable range functions and $(\G, \sigma)$-invariant spaces is due to Theorem \ref{thm-MI-rango}.

 Finally, when $V=S_\sigma^\G(\A)$ for some $\A\subseteq  L^2(\X)$, we know by Lemma \ref{lemma-connection-MI-spaces} that $Z_\sigma[V]$ is generated by $Z_\sigma[\A]$ as a $D$-MI space. Thus, Theorem \ref{thm-MI-rango} guarantees that the range function $J$ is $J(\alpha)=\clspan\{Z_{\sigma} [\phi](\alpha)\,:\,\phi\in\A\},$
a.e. $\alpha\in \widehat{\Gamma}$.
\end{proof}

When $\Z^n$ acts by translations on $\R^n$, it is known that there is a relationship  between the minimum number of functions that one needs to generate a shift-invariant space and the dimension of the spaces defined by the range function associated to it (see \cite[Proposition 4.1]{TW14}). This connection is also true for $(\G, \sigma)$-invariant spaces. In fact, we shall show that it is true at the general level of MI spaces.

Keeping the notation of Section \ref{sec-MI}, suppose that $M\subseteq L^2(\W, \calH)$ is a MI space with respect to $D$ generated by a finite number of functions in $L^2(\W, \calH)$. Define  the {\it length} of $M$ as
$$\ell(M):= \min\{n\in\N\colon \exists \,\,\Phi_1, \cdots,\Phi_n \in
M \textrm{ with } M=M_D(\Phi_1, \ldots,\Phi_n)\}.$$
The length of $M$ can be expressed in terms of the range function associated to $M$ as follows:

\begin{proposition}
 Let $M\subseteq L^2(\W, \calH)$ be a MI space  finitely generated and let $J$ be the range function associated to $M$ through Theorem \ref{thm-MI-rango}. Then,
 $$\ell(M)=\esssup_{\w\in\W} \dim(J(\w)).$$
\end{proposition}

\begin{proof}
 First put $\ell(M)=\ell$. Since $M$ can be generated  by a set of $\ell$ functions, Theorem \ref{thm-MI-rango} gives that
 $\dim(J(\w))\leq \ell$ for a.e. $\w\in\W$ and then $\esssup_{\w\in\W} \dim(J(\w))\leq \ell$. To see that indeed the equality must hold, it is enough to prove that $\{\w\in\W:\, \dim(J(\w))=\ell\}$ has positive  $\nu$-measure.

 Suppose towards a contradiction that  $\nu(\{\w\in\W:\, \dim(J(\w))=\ell\})=0$. Then,  $\dim(J(\w))\leq \ell-1$ for a.e. $\w\in\W$. Following the ideas of the proof of \cite[Proposition 4.1]{TW14}, we will construct a set of $\ell-1$ generators for $M$.

 Let $\{\Phi_1, \ldots, \Phi_\ell\}\in L^2(\W, \calH)$ be a set of generators for $M$. Then,  $J(\w)=\mathrm{span}\{\Phi_i(\w):\, 1\leq i\leq\ell\}$ for a.e. $\w\in\W$. Since  $\dim(J(\w))\leq \ell-1$ for a.e. $\w\in\W$, there must be  a vector of the set $\{\Phi_i(\w):\, 1\leq i\leq\ell\}$ belonging  to the space spanned by the others $\ell-1$ vectors. With this in mind, define the ($\nu$-measurable) function $f:\W\to\{1, \ldots, \ell\}$ as $f(\w)=\min\{1\leq i\leq\ell:\, J(\w)=\mathrm{span}\{\Phi_j(\w):\, 1\leq j\leq \ell, j\neq i\}\}$. Then, if $\W_i:=\{\w\in\W:\, f(\w)=i\}$, $\nu(\W\setminus\bigcup_{i=1}^\ell\W_i)=0$ and the union $\bigcup_{i=1}^\ell\W_i$ is disjoint. Note that if $\w\in\W_i$, then $J(\w)=\mathrm{span}\{\Phi_j(\w):\, 1\leq j\leq \ell, j\neq i\}$.

We define functions $\Psi_1, \ldots, \Psi_{\ell-1}\in L^2(\W, \calH)$ such that:\\
 on $\W_1$
 $$\Psi_j=\Phi_{j+1}, \textrm{ for } 1\leq j\leq\ell-1 $$
 and on $\W_i, \, 2\leq i\leq\ell$,
 $$\Psi_j=\begin{cases}
           \Phi_j &\textrm{ if } 1\leq j\leq i-1\\
           \Phi_{j+1} &\textrm{ if } i\leq j\leq \ell-1.
          \end{cases}
$$
Therefore, $J(\w)=\mathrm{span}\{\Psi_j(\w):\, 1\leq j\leq \ell-1\}$ for a.e. $\w\in\W$ and by Theorem \ref{thm-MI-rango}, $M=M_D(\Psi_1, \ldots, \Psi_{\ell-1})$ which is a contraction and the proposition is proven.
\end{proof}

For a $(\G, \sigma)$-invariant space $V$ which is finitely generated, we can define the {\it length of $V$} as
$\ell(V):=\min\{n\in\N\colon \exists \,\,\phi_1, \cdots,\phi_n \in V
 \textrm{ with } V=S_{\sigma}^\G(\phi_1, \ldots,\phi_n)\}$.
As an immediate consequence of the above result and Lemma \ref{lemma-connection-MI-spaces} we get:
\begin{corollary}\label{cor-lenght}
 Let $V\subseteq L^2(\X)$ be $(\G, \sigma)$-invariant space finitely generated and let $J$ be be the range function associated to $V$
  through Theorem \ref{thm-main}. Then,
 $$\ell(V)=\esssup_{\al\in\whG} \dim(J(\al)).$$
\end{corollary}

\section{Frame and Riesz systems  arising from quasi-$\G$-invariant actions}\label{sec:frame-riesz}
In this section we investigate necessary and sufficient conditions for  the system  $E_{\sigma}^{\G}(\A)$ to be a frame or a Riesz basis for $S_{\sigma}^{\G}(\A)$.

\begin{theorem}\label{thm-frames}
 Let $\A\subseteq L^2(\X)$ be a countable set and let $J$ be the  measurable range function associated to $V=S_{\sigma}^{\G}(\A)$. Then, the following two conditions are equivalent:
 \begin{itemize}
  \item [$(i)$] $E_{\sigma}^{\G}(\A)$ is a frame for $V$ with frame bound $0<A\leq B<+\infty$.
  \item [$(ii)$] For almost every $\alpha\in\widehat{\Gamma}$, the set $\{Z_{\sigma} [\phi](\alpha):\,\phi\in\A\}$ is a frame for $J(\alpha)$ with uniform frame bounds $A, B$.
 \end{itemize}

\end{theorem}
\begin{proof}
 By the isometry property of $Z_{\sigma}$ and by \eqref{quasi-periodic}   we have
 \begin{align}\label{eq-1}
  \sum_{\phi\in\A}\sum_{\g\in\Gamma}|\langle f, \Pi_{\sigma}(\g)\phi\rangle_{L^2(\X)}|^2&=
  \sum_{\phi\in\A}\sum_{\g\in\Gamma}|\langle Z_{\sigma}[f], Z_{\sigma}[\Pi_{\sigma}(\g)\phi]\rangle_{L^2(\widehat{\Gamma}, L^2(C))}|^2\nonumber\\
  &=\sum_{\phi\in\A}\sum_{\g\in\Gamma}|\int_{\widehat{\Gamma}}\langle Z_{\sigma}[f](\alpha), Z_{\sigma}[\phi](\alpha)\rangle_{L^2(C)}\overline{(\g, \alpha)}\,dm_{\widehat{\G}}(\alpha)|^2.
 \end{align}
Assume $(i)$ holds. For a fixed $\phi\in\A$, call $F(\alpha):=\langle Z_{\sigma}[f](\alpha), Z_{\sigma}[\phi](\alpha)\rangle_{L^2(C)}$ and note  that
$F\in L^1(\widehat{\Gamma})$.
Then,   by the Pontrjagin Duality
\cite[Theorem 1.7.2]{Rud62},
 $$\widehat{F}(\g)=\int_{\widehat{\Gamma}}\langle Z_{\sigma}[f](\alpha), Z_{\sigma}[\phi](\alpha)\rangle_{L^2(C)}\overline{(\g, \alpha)}\,dm_{\widehat{\G}}(\alpha)$$ and thus,  \eqref{eq-1} implies
$$\sum_{\g\in\Gamma}|\widehat{F}(\g)|^2<+\infty.$$
As a consequence we have that $\widehat{F}\in \ell^2(\G)$ and then $F\in L^2(\widehat{\G})$. By Plancherel's Theorem we then get,
\begin{align*}
\sum_{\g\in\Gamma}|\int_{\widehat{\Gamma}}\langle Z_{\sigma}[f](\alpha), &Z_{\sigma}[\phi](\alpha)\rangle_{L^2(C)}\overline{(\g, \alpha)}\,dm_{\widehat{\G}}(\alpha)|^2\,dm_\G(\g)\\&=\|\widehat{F}\|^2_{\ell^2(\Gamma)}=\|F\|^2_{L^2(\widehat{\Gamma})}\\
&=\int_{\widehat{\Gamma}}|\langle Z_{\sigma}[f](\alpha), Z_{\sigma}[\phi](\alpha)\rangle_{L^2(C)}|^2\,dm_{\widehat{\G}}(\alpha).
\end{align*}
From here the arguments leading to $(ii)$ are an effortless adaptation of those in \cite{Bow00} and \cite[Theorem 4.1]{CP10}, hence we leave them for the reader.

Conversely, from $(ii)$ and  since $Z_{\sigma}$ is an  isometry, it follows that for every $f\in S_{\sigma}^{\G}(\A)$
and for a.e. $\alpha\in\widehat{\Gamma}$,
\begin{equation}\label{eq-2}
 A\|Z_{\sigma}[f](\alpha)\|^2_{L^2(C)}\leq \sum_{\phi\in\A}|\langle Z_{\sigma}[f](\alpha), Z_{\sigma}[\phi](\alpha)\rangle_{L^2(C)}|^2
\leq B\|Z_{\sigma}[f](\alpha)\|^2_{L^2(C)}.
\end{equation}
Therefore, integrating over $\widehat{\Gamma}$, \eqref{eq-2} becomes
$$A\|Z_{\sigma}[f]\|^2_{L^2(\widehat{\Gamma}, L^2(C))}\leq \sum_{\phi\in\A}\int_{\widehat{\Gamma}}|\langle Z_{\sigma}[f](\alpha), Z_{\sigma}[\phi](\alpha)\rangle_{L^2(C)}|^2
\leq B\|Z_{\sigma}[f]\|^2_{L^2(\widehat{\Gamma}, L^2(C))}$$
and in particular, we have that $\alpha\mapsto\langle Z_{\sigma}[f](\alpha), Z_{\sigma}[\phi](\alpha)\rangle$ belongs to $L^2(\widehat{\Gamma})$ for each $\phi\in\A$. Thus, with similar arguments to those we used in the other implication, we conclude
$$\sum_{\phi\in\A}\sum_{\g\in\Gamma}|\langle f, \Pi_{\sigma}(\g)\phi\rangle_{L^2(\X)}|^2=
\sum_{\phi\in\A}\int_{\widehat{\Gamma}}|\langle Z_{\sigma}[f](\alpha), Z_{\sigma}[\phi](\alpha)\rangle_{L^2(C)}|^2,$$
and hence, the frame property of $E_{\sigma}^{\G}(\A)$ follows immediately.
\end{proof}

In \cite{HSWW10} it was proven that the unitary representation $\Pi_{\sigma}$ is dual integrable with bracket map given by $[\psi, \phi](\alpha)= \langle Z_{\sigma}[\psi](\alpha), Z_{\sigma}[\phi](\alpha)\rangle_{L^2(C)}$.
Then, if the set of generators $\A$ has only one single function, Theorem \ref{thm-frames}
reduces to \cite[Theorem 5.7]{HSWW10} when the unitary representation involved is $\Pi_{\sigma}$:

\begin{corollary}\label{cor-frames}
 Let $\psi\in\ L^2(\X)$ and define $\Omega_{\psi}=\{\alpha\in\widehat{\Gamma}:\, \|Z_{\sigma}[\psi](\alpha)\|^2_{L^2(C)}>0\}$.
 Then,  the following are equivalent:
 \begin{itemize}
  \item [$(i)$] $E_{\sigma}^{\G}(\psi)$ is a frame for $S_{\sigma}^{\G}(\psi)$ with frame bound $0<A\leq B<+\infty$.
  \item [$(ii)$] For almost every $\alpha\in\Omega_{\psi}$,
  $A\leq \|Z_{\sigma}[\psi](\alpha)\|^2_{L^2(C)}\leq B.$
 \end{itemize}
\end{corollary}

In the next theorem we state the analogous result to Theorem \ref{thm-frames} for Riesz bases.

\begin{theorem}\label{thm-riesz}
 Let $\A\subseteq L^2(\X)$ be a countable set and let $J$ be the  measurable range function associated to $V=S_{\sigma}^{\G}(\A)$. Then, the following two conditions are equivalent:
 \begin{itemize}
  \item [$(i)$] $E_{\sigma}^{\G}(\A)$ is a Riesz basis  for $V$ with  bound $0<A\leq B<+\infty$.
  \item [$(ii)$] For almost every $\alpha\in\widehat{\Gamma}$, the set $\{Z_{\sigma} [\phi](\alpha):\,\phi\in\A\}$ is a Riesz basis for $J(\alpha)$ with uniform bounds $A, B$.
 \end{itemize}
\end{theorem}
\begin{proof}
We proceed as in \cite[Theorem 4.3]{CP10} and \cite{Bow00}. Let $\{b_{(\g,\phi)}\}_{\g\in\Gamma, \phi\in\A}$ be a scalar sequence with finite support.  Then we  have,
 \begin{align}\label{eq:riesz}
  \|\sum_{\g\in\Gamma, \phi\in\A}&b_{(\g,\phi)}\Pi_\sigma(\g)\phi\|_{L^2(\X)}^2
  =\|\sum_{\g\in\Gamma, \phi\in\A}b_{(\g,\phi)}Z_\sigma[\Pi_\sigma(\g) \phi]\|^2\nonumber\\
  &=\int_{\widehat{\G}}\|\sum_{\phi\in\A}[\sum_{\g\in\Gamma}b_{(\g,\phi)}X_\g(\alpha)] Z_{\sigma}[\phi](\alpha)\|^2_{L^2(C)}\,dm_{\widehat{\G}}(\alpha).
 \end{align}
 For each $\phi\in\A$, let $P_\phi$ be the trigonometric polynomial $P_\phi=\sum_{\g\in\Gamma}b_{(\g,\phi)}X_\g$. Note that since $\{b_{(\g,\phi)}\}_{\g\in\Gamma, \phi\in\A}$ has finite support, only finitely many polynomials $P_\phi$ are non-zero. In particular,  for every $\alpha\in\widehat{\Gamma}$, the scalar sequence $\{P_\phi(\alpha)\}_{\phi\in\A}$ has finite support and moreover
 \begin{equation}\label{eq:norm-P-phi}
\int_{\widehat{\G}}\|P_\phi(\alpha)\|^2_{\ell^2(\A)}\,dm_{\widehat{\G}}(\alpha)=\sum_{\phi\in\A}\|P_\phi\|^2_{L^2(\widehat{\G})}
=\sum_{\phi\in\A}\sum_{\g\in\Gamma}|b_{(\g,\phi)}|^2,
\end{equation}
for a.e. $\alpha\in\widehat{\G}$.
\newpage
Suppose $(ii)$ holds.  Then, for $\{a_\phi\}_{\phi\in\A}=\{P_\phi(\alpha)\}_{\phi\in\A}$ we have
\begin{equation}\label{eq:b}
A \sum_{\phi\in\A}|P_\phi(\alpha)|^2 \leq \| \sum_{\phi \in \A} P_\phi(\alpha) Z_\sigma[\phi](\alpha)\|^2_{L^2(C)} \leq B \sum_{\phi\in\A}|P_\phi(\alpha)|^2,
\end{equation}
for a.e. $\al\in\widehat{\G}$.
Integrating over $\widehat{\G}$ in \eqref{eq:b} and using  \eqref{eq:riesz} and \eqref{eq:norm-P-phi} we get
$$A\sum_{\phi\in\A}\sum_{\g\in\Gamma}|b_{(\g,\phi)}|^2\leq \|\sum_{\g\in\Gamma, \phi\in\A}b_{(\g,\phi)}\Pi(\g)\phi\|^2
\leq B\sum_{\phi\in\A}\sum_{\g\in\Gamma}|b_{(\g,\phi)}|^2.$$
Hence, $(i)$ holds.

Conversely, suppose, towards a contradiction, that $(ii)$ fails. Then, by a similar argument as the one given in the proof of \cite[Theorem 4.1]{CP10}, there must exist  a measurable set $W\subseteq \widehat{\G}$ with $m_{\widehat{\G}}(W)>0$, a sequence with finite support
$a=\{a_\phi\}_{\phi\in\A}$ and $\varepsilon>0$  such that
\begin{equation}\label{eq:fail-1}
\| \sum_{\phi \in \A} a_\phi Z_\sigma[\phi](\alpha)\|^2_{L^2(C)} >(B+\varepsilon)\|a\|^2_{\ell^2(\A)}, \quad\forall\,\alpha\in W
\end{equation}
or
\begin{equation}\label{eq:fail-2}
\| \sum_{\phi \in \A} a_\phi Z_\sigma[\phi](\alpha)\|^2_{L^2(C)} <(A-\varepsilon)\|a\|^2_{\ell^2(\A)}, \quad\forall\,\alpha\in W.
\end{equation}

For each $\phi\in\A$ such that $a_\phi\neq0$, let $m_\phi:=a_\phi\chi_W\in L^{\infty}(\widehat{\G})$. By \cite[Lemma 4.4]{CP10} there exists a sequence  of trigonometric polynomials  $\{P_n^{\phi}\}_{n\in\N}$ such that $ P_n^{\phi}(\alpha)\to m_\phi(\alpha)$ when $n\to+\infty$ for a.e. $\alpha\in\widehat{\G}$ and $\|P^{\phi}_n\|_{L^{\infty}(\widehat{\G})}\leq C$ for all $n\in\N$  and some  $C>0$.
For every $n\in\N$, let $\{b_{(\g,\phi)}^n\}_{\g\in\Gamma, \phi\in\A}$ be the sequence of coefficient of $P^{\phi}_n$.
Since,  $E_\sigma(\A)$ is a Riesz sequence we have
\begin{equation}\label{eq:riesz-2}
A\sum_{\g\in\Gamma, \phi\in\A}|b_{(\g,\phi)}^n|^2\leq
\|\sum_{\g\in\Gamma, \phi\in\A}b_{(\g,\phi)}^n\Pi(\g)\phi\|_{L^2(\X)}^2\leq B\sum_{\g\in\Gamma,\phi\in\A}|b_{(\g,\phi)}^n|^2.
\end{equation}
Using  \eqref{eq:norm-P-phi} and \eqref{eq:riesz} we can rewrite \eqref{eq:riesz-2} as
\begin{equation}\label{eq:p_n}
A\|P_n^{\phi}\|^2_{L^2(\widehat{\G})}\leq\int_{\widehat{\G}}\|\sum_{\phi\in\A}P_n^{\phi}(\alpha) Z_\sigma[\phi](\alpha)\|^2_{L^2(C)}\,dm_{\widehat{\G}}(\alpha)\leq B\|P_n^{\phi}\|^2_{L^2(\widehat{\G})}.
\end{equation}
Note that \eqref{eq:p_n} extends to $m_\phi$ leading to
$$A\|m_\phi\|^2_{L^2(\widehat{\G})}\leq\int_{\widehat{\G}}\|\sum_{\phi\in\A}m_\phi(\alpha) Z_\sigma[\phi](\alpha)\|^2_{L^2(C)}\,dm_{\widehat{\G}}(\alpha)\leq B\|m_\phi\|^2_{L^2(\widehat{\G})}.$$
This last inequality contradicts \eqref{eq:fail-1}, \eqref{eq:fail-2} because by integrating over $\widehat{\G}$ in \eqref{eq:fail-1} and  \eqref{eq:fail-2} we get
$$\int_{\widehat{\G}}\|\sum_{\phi\in\A}m_\phi(\alpha) Z_\sigma[\phi](\alpha)\|^2_{L^2(C)}\,dm_{\widehat{\G}}(\alpha)
\geq(B+\varepsilon)\|m_\phi\|^2_{L^2(\widehat{\G})}$$ and
$$\int_{\widehat{\G}}\|\sum_{\phi\in\A}m_\phi(\alpha) Z_\sigma[\phi](\alpha)\|^2_{L^2(C)}\,dm_{\widehat{\G}}(\alpha)
\leq(A-\varepsilon)\|m_\phi\|^2_{L^2(\widehat{\G})}.$$
\end{proof}

As for the frame case, when $\A$  has only one element, Theorem \ref{thm-riesz} recovers \cite[Proposition 5.3]{HSWW10}.

\begin{corollary}\label{cor-riesz}
 Let $\psi\in\ L^2(\X)$. Then, they are equivalent
 \begin{itemize}
  \item [$(i)$]  $E_{\sigma}^{\G}(\psi)$ is a Riesz basis for $S_{\sigma}^{\G}(\psi)$ with  bounds $0<A\leq B<+\infty$.
  \item [$(ii)$] For almost every $\alpha\in\widehat{\G}$,
  $A\leq \|Z_{\sigma}[\psi](\alpha)\|^2_{L^2(C)}\leq B.$
 \end{itemize}
\end{corollary}

In the next theorem, we prove that each $(\G, \sigma)$-invariant space can be decomposed into an orthogonal sum of principal
$(\G, \sigma)$-invariant spaces. This is not new or even surprising so far, since it follows from the well-known fact that  every unitary representation is the orthogonal sum of cyclic representations (cf. \cite[Proposition 3.3]{Fol95}). But we shall actually prove more, namely  that  each principal $(\G, \sigma)$-invariant space involved in the decomposition can be chosen to be generated by a function $\psi$ such that $E_{\sigma}^{\G}(\psi)$ is a Parseval frame for $S_{\sigma}^{\G}(\psi)$.   This result is based in the Decomposition Theorem for MI spaces \cite[Theorem 2.6]{BR14} (see also \cite{Hel86}).

\begin{theorem}\label{thm-dec}
 Let $V\subseteq  L^2(\X)$ be a $(\G, \sigma)$-invariant space. Then, there exist a sequence of functions $\{\psi_n\}_{n\in\N}\subseteq  L^2(\X)$  such that $V$ can be decomposed into an orthogonal sum
$$V=\bigoplus_{n=1}^\infty S_{\sigma}^{\G}(\psi_n)$$
and, for each $n\in\N$,  $E_{\sigma}^{\G}(\psi_n)$ is a Parseval frame for $S_{\sigma}^{\G}(\psi_n)$.
\end{theorem}

\begin{proof}
 Since $V$ is a $(\G, \sigma)$-invariant space, by Lemma \ref{lemma-connection-MI-spaces}, $M:=Z_\sigma[V]$ is a MI space with respect to $D=\{X_\g\}_{\g\in\G}$. Thus, by \cite[Theorem 2.6]{BR14} $M$ can be decomposed into an orthogonal sum $M=\bigoplus_{n=1}^\infty M_n$ where each $M_n$ is a MI space and the range function associated to $M_n$ is $J_n(\al)=\mathrm{span}\{\Phi_n(\al)\}$  with $\Phi_n\in L^\infty(\widehat{\Gamma}, L^2(C))$ such that $\|\Phi_n(\al)\|_{L^2(C)}\in\{0,1\}$ for a.e. $\al\in\whG$.

 Since, $\whG$ has finite measure, it follows that $\Phi_n\in L^2(\widehat{\Gamma}, L^2(C))$ and $M_n=M_D(\Phi_n)$. Now, define $\psi_n\in L^2(\X)$ such that $Z_\sigma[\psi_n]=\Phi_n$.  Then, since $Z_\sigma$ is an isometric isomorphism and by Lemma \ref{lemma-connection-MI-spaces} we have that  $V=\bigoplus_{n=1}^\infty S_{\sigma}^{\G}(\psi_n)$.
 Moreover, since $\|\Phi_n(\al)\|_{L^2(C)}\in\{0,1\}$ for a.e. $\al\in\whG$, Corollary  \ref{cor-frames} implies that $E_{\sigma}^{\G}(\psi_n)$ is a Parseval frame for $S_{\sigma}^{\G}(\psi_n)$ and the theorem is proven.
\end{proof}

\begin{remark}
 If $\{\psi_n\}_{n\in\N}$ is a sequence as in Theorem \ref{thm-dec} for the $(\G, \sigma)$-invariant space $V$, it follows that $E_{\sigma}^{\G}(\A)$ is a Parseval frame for $V$, where $\A=\{\psi_n:\, n\in\N\}$. In particular this guarantees that every  $(\G, \sigma)$-invariant space $V$ has a frame of the form  $E_{\sigma}^{\G}(\A)$ for some set $\A$.   This was already known when the group involved acts by translations (see \cite{Bow00, BR14, CP10}).
\end{remark}

So far, we know that every $(\G, \sigma)$-invariant space has a frame of the form $E_{\sigma}^{\G}(\A)$ for some set $\A$.
But the statement is not true anymore if we replace frame for Riesz basis. That is, not every $(\G, \sigma)$-invariant space has a Riesz basis of the form $E_{\sigma}^{\G}(\A)$ for some set $\A$. We now  provide  an example of this situation.

\begin{example}
 Let $W_1\subseteq \whG$ be a measurable set with $0<m_{\whG}(W_1)<1/2$. In particular, we have $m_{\whG}(\whG\setminus W_1)>0$. Further, let $W_2\subseteq C$ be a measurable set with $0<\mu(W_2)<+\infty$ and consider the function $\Phi(\al)(x):=\chi_{W_1}(\al)\chi_{W_2}(x)$. It is immediately seen that $\Phi\in L^2(\whG, L^2(C))$ and that for all $\al\in\whG$, $\|\Phi(\al)\|^2_{L^2(C)}=\chi_{W_1}(\al)\mu(W_2)$. Define $\psi\in L^2(\X)$ by $Z_\sigma[\psi]=\Phi$ and put $V=S_{\sigma}^{\G}(\psi)$. Then, the following hold:
 \begin{itemize}
  \item [$(i)$] $E_{\sigma}^{\G}(\psi)$ is a frame for $V$ but  not a Riesz basis.
  \item [$(ii)$] If  $\A\subseteq L^2(\X)$ and $V=S_{\sigma}^{\G}(\A)$ then, $E_{\sigma}^{\G}(\A)$ cannot be a Riesz basis for $V$.
 \end{itemize}
Item  $(i)$ simply follows from Corollaries \ref{cor-frames} and \ref{cor-riesz}. To see $(ii)$,
suppose that $\A$ is a set of generator for $V$ and that $E_{\sigma}^{\G}(\A)$ is a Riesz basis for $V$. Then, by Theorem \ref{thm-main}, for a.e.  $\al\in\whG$ , $\{Z_\sigma[\phi](\al):\,\phi\in\A\}$ is a Riesz basis for $J(\al)$ where $J$ is the range function associated to $V$. Thus,
denoting by $|\A|$ the number of elements of $\A$, we have that $|\A|=\dim(J(\al))$ for a.e. $\al\in\whG$.  But, on the other hand,  $V$ has length 1 and then  $|\A|=1$ (cf. Corollary \ref{cor-lenght}).

Now, take $\phi\in V$ such that  $V=S_{\sigma}^{\G}(\phi)$. Then,
since $V$ is generated also by $\psi$ and  $E_{\sigma}^{\G}(\psi)$
is a frame for $V$ we have that $\phi=\sum_{\g\in\G}c_\g\Pi_\sigma
(\g)\psi$ for some sequence of coefficients
$\{c_\g\}_{\g\in\G}\in\ell^2(\G)$. By \eqref{quasi-periodic},
$Z_\sigma[\phi]=\left(\sum_{\g\in\G}c_\g X_\g\right)Z_\sigma[\psi]$
and then, $\{\al\in\whG:\,Z_\sigma[\phi](\al)\neq0\}\subseteq
\{\al\in\whG:\, Z_\sigma[\psi](\al)\neq0\}=W_1$. This
means that $Z_\sigma[\phi]$ vanishes on a set of positive measure of
$\whG$ (namely $\whG\setminus W_1$) and then condition $(ii)$ in
Corollary \ref{cor-riesz} is not satisfied for $\phi$, implying that
$E_{\sigma}^{\G}(\phi)$ cannot be a Riesz basis for $V$.

\end{example}

\section{Closed subgroups acting by translation}\label{sec:TI-spaces}
Let $G$ be a second countable LCA group and let $\G$ be a closed subgroup of $G$ which does not need to be discrete.
We devote this section to study closed subspaces of $L^2(G)$ that are invariant under translations in $\G$.
When $G/\G$ is compact, the structure of theses spaces has been analyzed first in \cite{CP10} for the case when  $\G$ is discrete, and recently in \cite{BR14} when $\G$ is not required to be discrete. The main results of \cite{BR14, CP10}
are a characterization of spaces invariant under translations in terms of range function using fiberization techniques.

Our aim is to show how these spaces can be characterized in terms of range functions without assuming that $G/\G$ is compact.   For this purpose, we shall work with  an appropriate non-discrete generalized Zak  transform $Z$ which will reduce to the mapping $Z_\sigma$ introduced in Section \ref{sec-Zak} when the action $\sigma:\G\times G\to G$ is $\sigma(\g,x)=\g+x$ and $\G$ is discrete.
The mapping $Z$ will be  the one that will replace the fiberization techniques of \cite{BR14, CP10}.
Since we only assume that $\G$ is closed, our results will be applicable to cases where those in \cite{BR14, CP10} are not.

\subsection{Setting}\label{sec-setting}
Let $G$ be a second countable LCA group. As it was pointed out in \cite{BR14} this is equivalent to say that $G$ (or $\wh G$) is $\sigma$-compact and metrizable.
The translation operator by an element $y\in G$ is denoted by $T_y:L^2(G)\to L^2(G)$ and is given by $T_y f(x)=f(x-y)$.

Let $\G\subseteq G$ be a closed subgroup of $G$. We fix $C$ a $m_G$-measurable section of the quotient $G/\G$. Its existence  is guaranteed by \cite[Lemma 1.1]{Mac52} or \cite[Theorem 1]{FG68}.

The classes in $G/\G$ are denoted by $[x]=x+\G$, for $x\in G$.
We define the {\it cross-section} mapping $\tau:G/\G\to C$ by $\tau([x])=[x]\cap C$. Through $\tau$ we  can carry over the topological and algebraic structure of $G/\G$ onto $C$. Therefore, in what follows we will consider $C$ as a topological space with the topology it inherits from $G/\G$ through $\tau$. Thus, $\tau$ is bijective continuous mapping with continuous inverse. Furthermore, we can define a measure $\mu$ on $C$ by
\begin{equation}\label{eq:def-mu}
\mu(E):=m_{G/\G}(\tau^{-1}(E))
\end{equation}
for every set $E$ in the Borel $\sigma$-algebra $\Sigma=\{E\subseteq C\,:\, \tau^{-1}(E) \, \textrm{ is }m_{G/\G}\textrm{-measurable}\}$. Since $m_{G/\G}$ is Borel regular and $\tau$ and $\tau^{-1}$ are continuous,  $\mu$ is also a Borel regular measure on $C$. Moreover, since $C$ is Borel measurable with respect to $m_G$, by \cite[Theorem 1]{FG68}, we have that $\{A\subseteq C:\, A \, \textrm{ is }m_{G}\textrm{-measurable}\}\subseteq \Sigma$. Measures constructed as in \eqref{eq:def-mu} are usually called {\it pushforward   or image measures}. For details on the subject we refer to \cite[Chapter 3, Section 3.6]{Bog07}.

We need the  following version of  the so-called Weil's  formula, \cite[Theorem (28.54)]{HR70}).

\begin{theorem}\label{thm-weil-for-sections}
 Let $G$ be an LCA group and $\G$ be a  closed subgroup of $G$. Consider $C\subseteq G$ a $m_G$-measurable section for the quotient $G/\G$ and let $\mu$ be the measure on $C$ given by \eqref{eq:def-mu}. Then,
 for every $f\in L^1(G)$ one has that
for $\mu$-almost $x\in C$, the function $\g\mapsto f(x+\g)$ is $m_\G$-measurable and belongs to $L^1(\G)$. Furthermore, the function $x\mapsto \int_\G f(x+\g)\,dm_\G(\g)$ is $\mu$-measurable, belongs to $L^1(C):=L^1(C, \mu)$ and, if $m_\G$ and $\mu$ are fixed, the Haar measure on $G$, $m_G$ can be chosen  so that
\begin{equation}\label{eq:wiel-for-sections}
\int_G f(y)\,dm_G(y)=\int_C\int_\G f(x+\g)\,dm_\G(\g)\,d\mu(x).
\end{equation}
\end{theorem}

\begin{proof}
 First note that for every $F\in L^1(G/\G)$, $F\circ\tau^{-1}\in L^1(C)$ and
\begin{equation}\label{eq:integration-mu}
\int_{G/\G}F([x])\,dm_{G/\G}([x])=\int_C F(\tau^{-1}(x))\,d\mu(x).
\end{equation}
This is in fact \cite[Theorem 3.6.1]{Bog07}, but it is easily seen
since, by definition of $\mu$, \eqref{eq:integration-mu} holds for
$F=\chi_A$ with $A\subseteq G/\G$ $m_{G/\G}$-measurable. Then, by
linearity, \eqref{eq:integration-mu} holds for simple functions in
$L^1(G/\G)$. Finally, using  standard arguments of measure theory
(cf. \cite[Theorems 10.3 and  10.20]{WZ77}), one can see that
\eqref{eq:integration-mu} extends to all $F\in L^1(G/\G)$. It is
also true that every function  in $ L^1(C)$ is of the
form $F\circ\tau^{-1}$ for some $F\in L^1(G/\G)$.

Now, if $f\in L^1(G)$, \cite[Theorem (28.54)]{HR70} implies that $F([x])=\int_\G f(x+\g)\,dm_\G(\g)$ belongs to $L^1(G/\G)$
and the Haar measures can be chosen in order to have $\int_G f(y)\,dm_G(y)=\int_{G/\G}F([x])\,dm_{G/\G}([x])$. Then, the result follows by applying \eqref{eq:integration-mu}.
\end{proof}

\begin{remark}\label{rem:measures}\noindent
\begin{enumerate}
\item When $\G$ is assumed to be discrete and countable  and $m_\G$ is chosen to be the counting measure, the measure $\mu$ in \eqref{eq:def-mu} is the restriction of $m_G$ to $C$. This follows from an easily adaptation of the proof of \cite[Lemma 2.10]{CP10}. In particular, we have that $0<m_G(C)$.
\item When $\G$ is not discrete, it could be that $0=m_G(C)$. Then,  $\mu\neq m_G|_C$. Indeed, if $\mu$
were $m_G|_C$, the right hand side of
\eqref{eq:wiel-for-sections} would be $0$ for every $f\in L^1(G)$
which is a contraction. As an example, consider
$\G=\R\times\{0\}\subseteq G=\R^2$. Then, $C=\{0\}\times\R$ and it
has zero Lebesgue measure. In this case, formula
\eqref{eq:wiel-for-sections} becomes Fubini's formula,
$\int_{\R^2}f(x,y)\,d(x,y)=\int_{\R}\left(\int_{\R}f(x,y)\,dx\right)\,dy$.
\end{enumerate}
\end{remark}

The annihilator of $\G$ is the (closed) subgroup $\G^*=\{\delta \in \widehat G: (\g,\delta)=1 \ \mbox{for all}\  \g\in \G\}$ of the dual group $\wh G$ of $G$.
Let $\W$ be a $m_{\wh G}$-measurable section of $\wh G/\G^*$.
The section  $\W$ of $\wh G/\G^*$ can be identified with $\whG$, through the bijective mapping  $\pi:\W\to\widehat{\G}$ given by
\begin{equation}\label{dual-H-Omega}
\pi(\w):\G\longrightarrow\C,\qquad
 \pi(\w)(\g)=(\g,\w).
\end{equation}
If $\tilde{\tau}$ is the cross-section mapping for the quotient $\wh G/\G^*$ and the section $\W$, then, $\pi\circ\tilde{\tau}:\wh G/\G^*\to\whG$ is an isomorphism which is in fact an homeomorphism (see \cite[Theorem 2.1.2]{Rud62}).
Furthermore, it can be proven that $m_{\wh G/\G^*}((\pi\circ\tilde{\tau})^{-1}(\cdot))$
defines a Haar measure on $\whG$, that we call $m_{\whG}$. In the same way we saw \eqref{eq:integration-mu}, one can show that for every $F\in L^1(\wh G/\G^*)$,
\begin{equation*}
 \int_{\wh G/\G^*}F([\xi])\,dm_{\wh G/\G^*}([\xi])=\int_{\whG} F((\pi\circ\tilde{\tau})^{-1}(\xi))\,dm_{\whG}(\xi).
\end{equation*}
Now,  we consider on $\W$ the measure $\nu$ defined by \eqref{eq:def-mu} but with $m_{\wh G/\G^*}$ and the cross-section mapping $\tilde{\tau}$ associated to $\wh G/\G^*$ and $\W$. Then, Theorem \ref{thm-weil-for-sections} and the above discussion prove that for every $G\in L^1(\W):=L^1(\W, \nu)$,
\begin{equation}\label{eq:Omega-gamma-dual}
 \int_{\W}G(\w)\,d\nu(\w)=\int_{\whG} G(\pi^{-1}((\cdot, \w))\,dm_{\whG}((\cdot, \w)),
\end{equation}
where $(\cdot, \w)$ are all the elements of $\whG$.

As we specified in Section \ref{sec:LCA-basics}, for each $\g\in\G$, the corresponding character on $\whG$ is denoted by $X_\g$. We then define for each $\g\in\G$, the functions  $\tilde{X}_\g:\W\to\C$ by $\tilde{X}_\g(\w):=X_\g\circ\pi^{-1}((\cdot, \w))=(\g, \w)$. In particular, by the uniqueness of the Fourier transform and \eqref{eq:Omega-gamma-dual}, we conclude that $D=\{\tilde{X}_\g\}_{\g\in\G}$ is a determining set for $L^1(\W)$. This shall be relevant in Subsection \ref{subsec:TI-spaces}.

\subsection{A non-discrete generalized Zak transform}\label{sec-non-discrete-Zak}
Let us keep the notation and the hypotheses on $G$ and  $\G$ of
Subsection \ref{sec-setting}. Throughout this section $\W$ is a
$m_{\wh G}$-measurable section for the quotient $\wh G/\G^*$ and $C$
is a $m_{G}$-measurable section for the quotient $ G/\G$.
 We regard $C$ and $\W$ as measure spaces with measures
$\mu$ and $\nu$ respectively, where both are given by
\eqref{eq:def-mu}. Furthermore. we consider the following normalization of the Haar measures involved: first, we fix $m_{\G}$ such that the Inversion formula \cite[Section 1.5]{Rud62} holds for  $m_{\G}$ and $m_{\whG}=m_{\wh G/\G^*}((\pi\circ\tilde{\tau})^{-1}(\cdot))$. In particular, we have that $m_{\G}$ and $m_{\whG}$ are dual measures, meaning that Plancherel's Theorem is true. Then, for $\mu$ and the already fixed $m_\G$, we choose $m_G$ such that Theorem \ref{thm-weil-for-sections} holds.

\begin{definition}\label{def:Zak-non-discrete}
Let $\psi\in L^1(G)$.  For $\w\in\W$ and $x\in C$ such that $\g\mapsto\psi(x-\g)\in L^1(\G)$ we define the {\it Zak transform} of $\psi$ as
$$Z[\psi](\w)(x):=\int_\G \psi(x-\g)(-\g,\w) \,dm_\G(\g).$$
\end{definition}
Note that, due to Theorem \ref{thm-weil-for-sections},  the above definition makes complete sense for all $\w\in\W$ and for $\mu$-a.e. $x\in C$. Moreover, $Z[\psi](\w)(x)=\calF_\G(\psi(x-\cdot))(\w)$
where $\calF_\G$ denotes the Fourier transform on $\G$.

Our aim is to show that $Z$ defines an isometric isomorphism between $L^2(G)$ and $L^2(\W, L^2(C))$.
We need the following intermediate lemma.  From now on,   $C_c(G)$ will stand for the set of continuous functions on $G$ with compact support.

\begin{lemma}\label{lemma:zak-non-discrete}
 For $\psi\in C_c(G)$, we have $Z[\psi]\in L^2(\W, L^2(C))$ and  $\|\psi\|_{L^2(G)}=\|Z[\psi]\|$.
 Moreover, $Z[T_\g\psi]=\tilde{X}_\g Z[\psi]$ for all $\g\in\G$.
\end{lemma}

\begin{proof}
First note that since $\psi\in C_c(G)$, $Z[\psi]$ is a continuous function on both variables.
Now, by Theorem \ref{thm-weil-for-sections}, we know that
\begin{equation}\label{eq:weil}
\int_G|\psi(y)|^2\,dm_G(y)=\int_C\int_\G |\psi(x-\g)|^2\,dm_\G(\g)\,d\mu(x).
\end{equation}
This implies that for $\mu$-a.e. $x\in C$, $\psi(x-\cdot)\in L^2(\G)$. By  Plancherel's Theorem and \eqref{eq:Omega-gamma-dual}, we have
\begin{equation}\label{eq:plancherel}
\int_\G |\psi(x-\g)|^2\,dm_\G(\g)=\int_\W |\calF_\G(\psi(x-\cdot))(\w)|^2\,d\nu(\w).
\end{equation}
Thus, we can replace \eqref{eq:plancherel} in \eqref{eq:weil} and then, since $Z[\psi]$ is continuous, we can apply Fubini's Theorem to obtain
$$\int_G|\psi(y)|^2\,dm_G(y)=\int_\W \int_C|\calF_\G(\psi(x-\cdot))(\w)|^2\,d\mu(x)\,d\nu(\w).$$
From here, we conclude that $Z[\psi]\in L^2(\W, L^2(C))$ and  $\|\psi\|_{L^2(G)}=\|Z[\psi]\|$.

To see that $Z[T_\g\psi]=\tilde{X}_\g Z[\psi]$, use \eqref{eq:translates-fourier} to obtain,
\begin{align*}
 Z[T_{\g}\psi](\w)(x)&=\mathcal{F}_\G[T_{\g}\psi(x-\cdot)](\w)\\
 &=\mathcal{F}_\G[\psi(x-(\g+\cdot))](\w)=(\g,\w)\mathcal{F}_\G[\psi(x-\cdot)](\w)\\
 &=(\g,\w)Z[\psi](\w)(x). \qedhere
 \end{align*}
\end{proof}

As a consequence of Lemma \ref{lemma:zak-non-discrete}, $Z:C_c(G) \to L^2(\W, L^2(C))$ defines an isometry that can be extended by density to the whole $L^2(G)$. We keep on denoting by $Z$ the extension of $Z:C_c(G) \to L^2(\W, L^2(C))$ to $L^2(G)$. As we shall show in the upcoming result, $Z$ turns out to be an isomorphism.

\begin{proposition}\label{prop:zak-non-discrete}
 Let $Z: L^2(G) \to L^2(\W, L^2(C))$ be the extension to $L^2(G)$ of the mapping  of Lemma \ref{lemma:zak-non-discrete}.
 Then, $Z$ is an isometric isomorphism satisfying $Z[T_\g\psi]=\tilde{X}_\g Z[\psi]$ for all $\g\in\G$ and all $\psi\in L^2(G)$.
\end{proposition}
\vspace{-2ex}
\begin{proof}
We only need to prove that $Z$ is onto. This will require to deal with some technical issues but roughly speaking, the idea is to show that there is a dense set of $L^2(\W, L^2(C))$ contained in the image of $Z$.
We structure the proof on  steps.

{\bf Step 1}: \noindent
Consider the set $\calD:=\{\Phi(\w;x)=\calF_\G(\Psi(\cdot;x))(\w):\,\,\Psi\in C_c(\G\times C)\}$. Then, by Plancherel's Theorem, $\calD\subseteq C(\W \times C)\cap L^2(\W \times C)$. We shall show now that $\calD$ is dense in $L^2(\W \times C)$. Towards this end, take $\Phi\in L^2(\W \times C)$ and $\varepsilon>0$. Now, choose $\Phi_1(\w;x)=\sum_{j=1}^n \alpha_jf_j(\w)g_j(x)$ with $f_j\in L^2(\W)$, $g_j\in L^2(C)$ and $\alpha_j\in\C$ for each  $1\leq j\leq n$ such that $\|\Phi-\Phi_1\|_{L^2(\W\times C)}<\varepsilon/2$. This can be done since the set $\textrm{span}\{f.g:\, f\in L^2(\W), \, g\in L^2(C)\}$  is dense in $ L^2(\W\times C)$ (e.g. \cite[pp. 51]{RS80}).

Put $\Psi_1(\g;x)=\sum_{j=1}^n \alpha_j\calF_\G^{-1}(f_j)(\g)g_j(x)$. Then, $\Psi_1\in
L^2(\G\times C)$ and, due to the density of $C_c(\G\times C)$ in
$L^2(\G\times C)$, we can choose $\Psi\in C_c(\G\times C)$ with
$\|\Psi_1-\Psi\|_{L^2(\G\times C)}<\varepsilon/2$. If we call
$\Phi_2(\w;x)=\calF_\G(\Psi(\cdot;x))(\w)\in\calD$, then, by
Plancherel's and  Fubini's  Theorems we have that
$\|\Phi_1-\Phi_2\|_{L^2(\W\times C)}=\|\Psi_1-\Psi\|_{L^2(\G\times
C)}<\varepsilon/2$. Therefore, $\|\Phi-\Phi_2\|_{L^2(\W\times
C)}<\varepsilon$ and consequently, $\calD$ is dense in $L^2(\W
\times C)$. Finally note that, since $\wh G/\G^*$ is
$\sigma$-compact, so is $\W$. Then, by Proposition
\ref{prop:isom-L-2}, $L^2(\W, L^2(C))=L^2(\W \times C)$ and we have
that $\calD$ is dense in $L^2(\W, L^2(C))$.

{\bf Step 2}: \noindent
We now prove that $\calD$ is contained in the range of $Z$. Let $\Phi\in\calD$, $\Phi(\w;x)=\calF_\G(\Psi(\cdot;x))(\w)$ with $\Psi\in C_c(\G\times C)$. Consider the mapping $y\mapsto(\tau([y])-y; \tau([y]))$ from $G$ to $\G\times C$, where $\tau$ is the cross-section mapping for $G/\G$ and $C$, and  define the function $\psi$ on $G$ as $\psi(y):=\Psi(\tau([y])-y; \tau([y]))$. Since $\Psi\in C_c(\G\times C)$ and $\tau([\cdot]): (G, m_G)\to (C, \mu)$ is measurable,  to see that $\psi$ is $m_G$-measurable it is enough to prove that $\xi(y):= \tau([y])-y$ is a measurable function from $(G, m_G)$ to $(\G, m_\G)$. Observe first that 
by the discussion after \eqref{eq:def-mu} in Subsection \ref{sec-setting}, we can conclude that 
$\tau([\cdot]):G\to C\subseteq G$ is a  measurable function when regarding $C$ with the Haar measure of $G$, $m_G$. Thus, $\xi: G\to \G\subseteq G$ is also  measurable when we consider in $\G$ the measure $m_G$. Note now, that since the topology in $\G$ is the relative topology as a subgroup of $G$, any open set of $\G$, $U$ is the intersection of an open set $V$ of $G$ with $\G$, i.e. $U=V\cap\G$. It follows then, that open sets of $\G$ are $m_G$-measurable and therefore, $\beta(\G)\subseteq \beta(G)$ where $\beta(\G)$ (resp. $\beta(G)$) denotes the Borel $\sigma$-algebra in $\G$ (resp. in $G$). In particular, this implies that $\xi$ is a measurable function from $(G, m_G)$ to $(\G, m_\G)$ as we wanted to prove. 
Moreover, $\psi$ is bounded because so is $\Psi$ and hence, $\psi$ is locally integrable on $G$.
Now, for every compact set $F\subseteq G$, we can apply \eqref{eq:wiel-for-sections} to $\chi_F|\psi|\in L^1(G)$ to obtain,
\begin{align}\label{eq:psi-integrable}
 \int_G \chi_F(y)|\psi(y)|\,dm_G(y)&=\int_C\int_\G \chi_F(x+\g)|\psi(x+\g)|\,dm_\G(\g)\,d\mu(x)\nonumber\\
 &=\int_C\int_\G \chi_F(x+\g)|\Psi(-\g;x)|\,dm_\G(\g)\,d\mu(x)\nonumber\\
 &\leq \int_C\int_\G |\Psi(\g;x)|\,dm_\G(\g)\,d\mu(x)<+\infty.
\end{align}
Since $G$ is $\sigma$-compact, $G=\bigcup_{n\in\N} F_n$ where $F_n\subseteq G$ is compact for every $n\in \N$ and $F_n\subseteq F_{n+1}$. Thus, since $\chi_{F_n}|\psi|\to|\psi|$ $m_G$-a.e. on $G$, by \eqref{eq:psi-integrable} and \cite[Corollary 10.30]{WZ77} we have that $\psi\in L^1(G)$.  With the same argument applied to $|\psi|^2$, we also see that $\psi\in L^2(G)\cap L^1(G)$.

Finally,
\begin{align*}
 Z[\psi](\w)(x)&=\int_\G \psi(x-\g)(-\g,\w) \,dm_\G(\g)\\
 &=\int_\G \Psi(\g;x)(-\g,\w) \,dm_\G(\g)\\
 &=\calF_\G(\Psi(\cdot;x))(\w)=\Phi(\w;x).
\end{align*}
Therefore, $\calD$ is contained in the range of $Z$ as we wanted to prove.

{\bf Step 3}: \noindent
Since the range of $Z$ contains a dense set and it  is closed, $Z$ must be onto.
\end{proof}

\begin{remark}\noindent
\begin{enumerate}
 \item Suppose that  $\G$ is discrete and countable. Then, since the Haar measure $m_G$ on $G$ is invariant under translations, the action  $\sigma:\G\times G\to G$, $\sigma(\g,x)=\g+g$ is a quasi-$\G$-invariant action with $J_\sigma(\g,x)\equiv1$.
 Note that the unitary representation associated to $\sigma$ is $\Pi_\sigma(\g)= T_\g$. Moreover, the section $C$ of $G/\G$ is a tiling set for $\sigma$ and then we are in a setting where the results of Section \ref{sec-Zak} can be applied. Then, if $Z_\sigma$ is as in Definition \ref{def-tau-sigma} and $Z$ as in Definition \ref{def:Zak-non-discrete}, we have $Z_\sigma=Z$.
 \item The construction of the mapping $Z$ was proposed by  Weil in \cite{Wei64}. In that work, the author
 works with the quotients instead of sections and
describes the range of $Z$ in a different way than here. To our
purpose, it is very important to exactly know that the range of $Z$
is a vector valued space since this is what allows us to connect
closed subspaces of $L^2(G)$ that are invariant under translations
with  MI spaces.
\end{enumerate}
\end{remark}

\subsection{Subspaces invariant under translations}\label{subsec:TI-spaces}
As we already said, we want to give a characterization in terms of range functions of closed subspaces of $L^2(G)$ that are invariant under translations in $\G$, where $G$ is a second countable  LCA group and $\G$ is a closed subgroup of $G$. To be precise, we say that a closed subspace $V$ of $L^2(G)$  is {\it invariant under translation in $\G$}, or {\it $\G$-invariant} for short, if $T_\g V\subseteq V$ for all $\g\in\G$.  As usual, we say that $V$ is generated by an at most countable  set $\A\in L^2(G)$ when $V=S^\G(\A):=\clspan\{T_\g\phi:\,\g\in\G, \,\phi\in\A\}$.

If $Z: L^2(G) \to L^2(\W, L^2(C))$ is the isomorphism of Proposition
\ref{prop:zak-non-discrete}, then  $Z$ turns  $\G$-invariant spaces
in $L^2(G)$ into  MI spaces in $L^2(\W, L^2(C))$ with respect to the
determining set $D=\{\tilde{X}_\g\}_{\g\in\G}$ and vice versa. This
fact can be proven in an analogous way to that used to show Lemma
\ref{lemma-connection-MI-spaces}. As a consequence, the following
characterization for $\G$-invariant spaces holds. Its proof is a
straightforward adaptation of that of Theorem \ref{thm-main} and
therefore we leave it for the reader.

\begin{theorem}\label{thm-TI-characterization}
Let $V\subseteq L^2(G)$ be a closed subspace and $Z$ be  the mapping of Proposition \ref{prop:zak-non-discrete}.
Then, the following conditions are equivalent:
\begin{enumerate}
 \item $V$ is a $\G$-invariant  space;
 \item there exists  $J:\W\to \{\textrm{closed subspaces of  }L^2(C)\}$, a measurable  range function, such that
$$V=\big\{f\in L^2(G):\,\, Z[f](\w)\in J(\w)\,,\,\nu-\textrm{a.e.}\,\,\w\in \W\big\}.$$
\end{enumerate}

Identifying  range functions which are equal almost everywhere, the correspondence between
$\G$-invariant spaces and  measurable range functions is one to one and onto.

Moreover, if $V=S^{\G}(\A)$ for some countable subset $\A$ of $L^2(G)$, the measurable  range function $J$
associated to $V$ is given by
$$J(\w)=\clspan\{Z [\phi](\w)\,:\,\phi\in\A\}, \qquad \nu-\textrm{a.e. }\w\in \W.$$
\end{theorem}

\begin{remark}
Theorem \ref{thm-TI-characterization} characterizes $\G$-invariant spaces in terms of range function in the same spirit as in \cite[Theorem 3.8]{BR14}. The main difference with the results in \cite{BR14} is that we do not require $\G$ to be co-compact (i.e. $G/\G$ compact). The co-compact assumption on $\G$ guaranties that $\G^*$ is discrete and this is what is crucial to define the mapping $\T$ of \cite[Proposition 3.7]{BR14} (see also \cite[Proposition 3.3]{CP10}).
Working with the isomorphism  $Z$ instead of $\T$ allows as to withdraw the co-compactness on $\G$ and as a consequence, our results are applicable to cases where those in \cite{BR14} are not. To give a simple example, consider $G=\R^3$ and $\G=\Z\times\R\times\{0\}$. Then $\G$ is closed but $G/\G\approx\Tor\times\{0\}\times\R$ is not compact.
Furthermore, when $\G$ is discrete (but not necessarily co-compact) Theorem \ref{thm-TI-characterization} is still more general than \cite[Theorem 3.10]{CP10}. For instance, it covers the (simple) case when $\G=\Z\times\{0\}\subseteq G=\R^2$ which does not fit in the setting of \cite[Theorem 3.10]{CP10} because $G/\G\approx\Tor\times\R$ is not compact.
\end{remark}

\begin{remark}
As a consequence of Theorem \ref{thm-TI-characterization} we recover
Corollary 3.9 in \cite{BR14}. That is, for a second countable LCA
group $G$, a closed subspace $V$ of $L^2(G)$ is invariant under all
translations by $G$ if and only if there exists a measurable set
$E\subset \widehat G$ such that $V= \{f\in L^2(G): \supp \widehat f
\subset E\}.$ Indeed, take $\Gamma = G$ in Theorem
\ref{thm-TI-characterization}; then $L^2(C) = L^2(\{e\})= \C$,
$\Omega = \widehat G,$ each $J(\omega)$ is a closed subspace of $\C$
and hence $J(\omega) = \C$ or $\{0\}$, and $Z[f](\omega)(e) =
\mathcal{F} f (-\omega)\,.$ Thus, it is enough to take $E = \{\omega
\in \widehat G : J(-\omega) = \C\}$.
\end{remark}

As one may expect, it is also possible to characterize frame and Riesz conditions of the system $E^\G(\A):=\{T_\g\phi:\,\g\in\G, \,\phi\in\A\}$ in terms of $Z$. Since
$E^\G(\A)$ is indexed by the - in general - non-discrete set $\G\times\A$, the frame and Riesz properties have to be regarded in the continuous sense as in Definitions 5.2 and 5.3 in \cite{BR14}. Then, we have:

\begin{theorem}\label{thm-continuous-frame-riesz}
 Let $\A\subseteq L^2(G)$ be a countable set and let $J$ the a measurable range function associated to $V=S^{\G}(\A)$ through Theorem \ref{thm-TI-characterization}. Then, the following two conditions are equivalent:
 \begin{itemize}
  \item [$(i)$] $E^{\G}(\A)$ is a continuous frame (resp., continuous Riesz basis)  for $V$ with  bound $0<A\leq B<+\infty$.
  \item [$(ii)$] For $\nu$-almost every $\w\in\W$, the set $\{Z[\phi](\w):\,\phi\in\A\}$ is a frame (resp.,  Riesz basis) for $J(\w)$ with uniform bounds $A, B$.
 \end{itemize}
\end{theorem}

The proof of this result is, once again, a straightforward adaptation of the proof of \cite[Theorem 5.1]{BR14} and therefore we omit it. However, we do want to emphasize what was noticed in \cite{BR14}: first, despite item  $(i)$ in Theorem \ref{thm-continuous-frame-riesz} involves the concept of continuous frames (resp., Riesz bases), the frame and Riesz basis notions in $(ii)$ are the usual ones. Second,  with the same strategies than those used in the proof of Theorem 5.1 for Riesz bases in  \cite{BR14}, it can be seen that  both item  $(i)$ and item $(ii)$ in Theorem \ref{thm-continuous-frame-riesz} for Riesz bases imply that $\G$ is discrete. Then,  the result reduces to Theorem \ref{thm-riesz} where the action involved is
$\sigma:\G\times G\to G$, $\sigma(\g,x)=\g+x$.

\subsection{Connection with the so-called Fiberization mapping}
As we already pointed out, when $\G$ is assumed to be co-compact, $\G$-invariant spaces are characterized in terms of range functions either by Theorem \ref{thm-TI-characterization} or
\cite[Theorem 3.8]{BR14}. What we shall investigate now, is the relationship of the main tools used in those theorems, mainly $Z$, the  non-discrete generalized Zak transform, and $\T$ the fiberization mapping of \cite[Proposition 3.7]{BR14}.

Let  $G$, $\G$, $\wh G$, $\G^*$, $\W$ and $C$  be as in  Section \ref{sec-setting} with their respective measures and suppose, additionally,  that $G/\G$ is compact. Then, $\G^*$ is discrete and countable (cf. \cite[Theorem 24.15]{HR79}). By Remark \ref{rem:measures}, we have that $\nu$, the measure on $\W$ given by \eqref{eq:def-mu}, is the restriction of $m_{\wh G}$ to $\W$.  As we identified $\W$ with $\whG$, we can identify $C$ with the dual group of $\G^*$.

The isometry $\T$ proposed in \cite{BR14, CP10} is the mapping $\T:L^2(G)\to L^2(\W, \ell^2(\G^*))$ defined by
$$\T f(\w)=\{\widehat{f}(\w+\delta)\}_{\delta\in \G^*}.$$

\begin{proposition}\label{tau-vs-tau}
 For $f\in \mathcal{C}_c(G)$ such that $\widehat{f}\in L^1(\widehat{G})$ it holds that
 $$\calF_{\G^*}(\T f(\w))(x)=(x,\w)Z_{\sigma}[f](-w)(-x),$$
 for $x\in C$ and $\w\in\W$, where $\calF_{\G^*}$ denotes the Fourier transform with respect to the group $\G^*$.
\end{proposition}

\begin{proof}
 Fix $x\in C$ and $\w\in\W$. Denoting by $M_{\w}$ the modulation operator by $\w$, i.e $M_{\w}\phi(\cdot)=(\cdot, \w)\phi(\cdot)$, we can rewrite $Z_{\sigma}[f](-w)(-x)$ as
 \begin{align*}
  Z_{\sigma}[f](-\w)(-x)&=\int_{\G}f(-x+\g)(\g,-\w)\, dm_\G(\g)\\
  &=(x,-\w)\int_{\G}T_{x}M_{-\w}f(\g)\, dm_\G(\g).
  \end{align*}
Since $f\in \mathcal{C}_c(G)$ and  $\widehat{f}\in L^1(\widehat{G})$, the same holds for $T_{x}M_{-\w}f$. This is, $T_{x}M_{-\w}f\in \mathcal{C}_c(G)$ and  $\widehat{T_{x}M_{-\w}f}\in L^1(\widehat{G})$. Therefore, by the Poisson Formula \cite[Theorem 5.5.2]{Rei68},
\begin{equation}\label{eq:poisson}
\int_{\G}T_{x}M_{-\w}f(\g)\, dm_\G(\g)=\sum_{\delta\in \G^*}\widehat{T_{x}M_{-\w}f}(\delta)
=\sum_{\delta\in \G^*}\widehat{f}(\w+\delta)(-x,\delta).
\end{equation}
Since the right hand side of \eqref{eq:poisson} is the Fourier
transform of $\T f(\w)$ at $x$ with respect to the
group $\Gamma^*$, the result follows.
\end{proof}

As a consequence of  Plancherel's Theorem, we immediately obtain:
\begin{corollary}
 If $f$ and $g$ are as in Proposition \ref{tau-vs-tau} then,
 $$\langle \T f(\w), \T g(\w)\rangle_{\ell^2(\G^*)}=\langle Z_{\sigma} [f](-\w), Z_{\sigma}[g](-\w)\rangle_{L^2(C)}.$$
\end{corollary}

\appendix
\section{}\label{sec:appendix}
In this section we shall prove two results concerning vector valued
spaces. The first one says that, when $\W$ has finite measure and
$\calH$ is any separable Hilbert space, the set of simple functions
of $L^2(\Omega, \mathcal{H})$ is dense. The second one shows that,
when $\W$ is $\sigma$-finite and the Hilbert space $\calH$ is an
$L^2$-space, this is $\calH=L^2(C)$ for some measure space $(C,
\mu)$, then $L^2(\W, L^2(C))=L^2(\W\times C)$.  We
believe that both results are known, but since we could not find an
explicit reference for them, we decided to provide a proof. We will
make use of the following intermediate result. Its proof can be
found in \cite[Proof of Theorem 2.1, Chapter II]{DU77}.

\begin{lemma}\label{lemma-quasi-simple}
Let $(\W, \nu)$ be a measure space such that $\nu(\W)<+\infty$ and let
$\calH$ be a separable Hilbert space. If $\Phi:\Omega\to \calH$
is a measurable function, then, for
every $\varepsilon>0$ there exist a partition of  $\Omega$ into
measurable sets, $\{B_n\}_{n\in \N}$ and a sequence
$\{x_n\}_{n\in\N}\subseteq \calH$ such that
$\Psi:=\sum_{n=1}^{\infty}x_n\chi_{B_n}$ satisfy
$\|\Phi(\w)-\Psi(\w)\|_\calH<\varepsilon$ for a.e. $\w\in \W$.
\end{lemma}

\begin{proposition}\label{prop:simples-are-dense}
Suppose that $\nu(\W)<+\infty$. Then for each $\Phi\in L^2(\Omega, \mathcal{H})$ there exists a sequence $\{\Phi_n\}_{n\in\N}\subseteq L^2(\Omega, \mathcal{H})$ of simple functions such that $\lim_{n\to\infty}\|\Phi-\Phi_n\|=0$.
\end{proposition}

\begin{proof}
Fix $\Phi\in L^2(\Omega, \mathcal{H})$. For each $n\in\N$, by Lemma \ref{lemma-quasi-simple} we can choose $\Psi_n=\sum_{j=1}^{\infty}x_j^n\chi_{B_j^n}$ with   $\{B_j^n\}_{j\in \N}$ being a partition of $\W$ in measurable sets and $\{x_j^n\}_{j\in\N}\subseteq \calH$ such that
$\|\Phi(\w)-\Psi_n(\w)\|_\calH<\frac1{n}$ for a.e. $\w\in \W$.
Since $\Phi\in L^2(\Omega, \mathcal{H})$, we have that $\Psi_n\in L^2(\Omega, \mathcal{H})$ for all $n\in\N$. Indeed,
for a.e. $\w\in\W$, $\|\Psi_n(\w)\|_\calH\leq \frac1{n}+\|\Phi(\w)\|_\calH\leq 2\max\{\frac1{n}, \|\Phi(\w)\|_\calH\} $ and then
\begin{equation}\label{eq:norm-2}
 \|\Psi_n(\w)\|_\calH^2\leq 4\max\{\frac1{n^2}, \|\Phi(\w)\|^2_\calH\}\leq 4(\frac1{n^2}+ \|\Phi(\w)\|^2_\calH).
\end{equation}
Now, integrating over $\W$ in \eqref{eq:norm-2} we get
$\|\Psi_n\|^2=\int_{\W}\|\Psi_n(\w)\|_\calH^2\,d\nu(\w)\leq 4(\frac1{n^2}\nu(\W)+ \|\Phi\|^2)<+\infty$.

Note that, for every $n\in\N$, $\|\Psi_n\|^2=\sum_{j=1}^{\infty}\|x_j^n\|^2_\calH\nu(B_j^n)$ and then we can choose $N(n)$ large enough so that
\begin{equation}\label{eq:tail}
\sum_{j=N(n)+1}^{\infty}\|x_j^n\|^2_\calH\nu(B_j^n)=\int_{\bigcup_{j=N(n)+1}^{\infty}B_j^n}\|\Psi_n(\w)\|_\calH^2\,d\nu(\w)<\frac1{n}.
\end{equation}

Define $\Phi_n=\sum_{j=1}^{N(n)}x_j^n\chi_{B_j^n}$ and let us see that $\Phi_n\to\Phi$, as $n\to+\infty$ in $L^2(\Omega, \mathcal{H})$.
Proceeding as before and using \eqref{eq:tail}, we have
\begin{align*}
  \|\Phi_n-\Phi_n\|^2&=\int_{\W}\|\Phi_n(\w)-\Phi(\w)\|_\calH^2\,d\nu(\w)\\
  &\leq 4\left(\int_{\W}\|\Phi_n(\w)-\Psi_n(\w)\|_\calH^2\,d\nu(\w)+\int_{\W}\|\Psi_n(\w)-\Phi(\w)\|_\calH^2\,d\nu(\w)\right)\\
  &\leq 4\left(\int_{\bigcup_{j=N(n)+1}^{\infty}B_j^n}\|\Psi_n(\w)\|_\calH^2\,d\nu(\w)+\int_{\W}\frac1{n^2}\,d\nu(\w)\right)\\
  &\leq 4(\frac1{n}+\frac{\nu(\W)}{n^2}),
\end{align*}
and the result follows.
\end{proof}

As it was pointed out to us by the referee, the next proposition is a generalization of \cite[Theorem II.10 (c)]{RS80}. 

\begin{proposition}\label{prop:isom-L-2}
 Let $(\W, \nu)$ be a $\sigma$-finite measure space and $(C, \mu)$ be a measure space such that $L^2(C)$ is separable.
 Then, $L^2(\W, L^2(C))=L^2(\W\times C)$.
\end{proposition}
\begin{proof}
 Clearly, by Fubini's Theorem, $L^2(\W \times C)\subseteq L^2(\W, L^2(C))$. The  claim $L^2(\W, L^2(C))=L^2(\W \times C)$ will follow from the fact that $\W$ is $\sigma$-finite and Proposition \ref{prop:simples-are-dense}. To see this, first write  $\W=\bigcup_{n\in\N} \W_n$ where, for every $n\in \N$, $\W_n\subseteq \W$ is of finite $\nu$-measure   and $\W_n\subseteq \W_{n+1}$.

For every $n\in \N$, each simple function in $L^2(\W_n, L^2(C))$
(simple in the sense of vector valued functions) is a $(\nu\times
\mu)$-measurable function from $\W_n \times C$ to $\C$. Therefore,
if $\calD_n=\{\textrm{simple functions of } L^2(\W_n, L^2(C))\}$ we
then have that $\calD_n\subseteq L^2(\W_n \times C)\subseteq
L^2(\W_n, L^2(C))$. Now, since the norm in $L^2(\W_n, L^2(C))$
 restricted  to  $L^2(\W_n \times C)$ is the norm on
$L^2(\W_n \times C)$ and $L^2(\W_n \times C)$ is closed, we conclude
from Proposition \ref{prop:simples-are-dense} that $L^2(\W_n \times
C)= L^2(\W_n, L^2(C))$.

Finally, take $\Phi\in L^2(\W, L^2(C))$ and put $\Phi_n:=\chi_{\W_n}\Phi$. Then, for all $n\in\N$, $\Phi_n\in L^2(\W_n, L^2(C))=L^2(\W_n \times C)$ and in particular, $\Phi_n$ is  a $(\nu\times \mu)$-measurable function on the product space $\W_n\times C$. Thus,  since  $\lim_{n\to\infty}\Phi_n(\w)(x)=\Phi(\w)(x)$ for  $(\nu\times \mu)$-a.e. $(\w;x)\in\W\times C$,  $\Phi$ is  a $(\nu\times \mu)$-measurable function on the product space $\W\times C$ and therefore, $\Phi\in L^2(\W\times C)$.
\end{proof}

\end{document}